\pdfoutput=1
\documentclass[12pt,reqno]{amsart}
\usepackage[lmargin=1.1in,rmargin=1.1in,tmargin=1in,bmargin=1in,headsep=20pt]{geometry}

\usepackage{yfonts}
\usepackage{amsthm,amsfonts,amssymb}

\usepackage{mathrsfs}
\usepackage[scaled=1.15,mathscr]{urwchancal}

\usepackage{enumitem}
\usepackage[colorlinks]{hyperref}    
\hypersetup{linkcolor=blue, citecolor=blue}

\newcommand{\keep}[1]{{\binoppenalty=10000\relpenalty=10000 #1}}
\renewcommand{\epsilon}{\varepsilon}

\def\Xint#1{\mathchoice
   {\XXint\displaystyle\textstyle{#1}}%
   {\XXint\textstyle\scriptstyle{#1}}%
   {\XXint\scriptstyle\scriptscriptstyle{#1}}%
   {\XXint\scriptscriptstyle\scriptscriptstyle{#1}}%
   \!\int}
\def\XXint#1#2#3{{\setbox0=\hbox{$#1{#2#3}{\int}$}
     \vcenter{\hbox{$#2#3$}}\kern-.5\wd0}}
\def\dint{\Xint-}

\renewcommand{\Re}{\operatorname{Re}}

\renewcommand{\Im}{\operatorname{Im}}

\let\originalleft\left
\let\originalright\right
\renewcommand{\left}{\mathopen{}\mathclose\bgroup\originalleft}
\renewcommand{\right}{\aftergroup\egroup\originalright}

\newcommand{\pfrac}[2]{\left(\frac{#1}{#2}\right)}

\newcommand{\pfr}[2]{\left(\frac{#1}{#2}\right)}
\newcommand{\afrac}[2]{\left|\frac{#1}{#2}\right|}
\newcommand{\fr}[2]{\frac{#1}{#2}}
\newcommand{\nfr}[2]{#1/#2}
\newcommand{\tf}[2]{\tfrac{#1}{#2}}
\newcommand{\tfr}[2]{\tfrac{#1}{#2}}
\newcommand{\rfr}[2]{\tfrac{1}{#2}#1}

\newcommand{\tops}[1]{\texorpdfstring{#1}{}}
\newcommand{\qqand}{\qquad\text{and}\qquad}

\newtheorem{theorem}{Theorem}
\newtheorem{proposition}[theorem]{Proposition}
\newtheorem{lemma}[theorem]{Lemma}
\newtheorem{corollary}[theorem]{Corollary}
\newtheorem*{theorem*}{Theorem}

\theoremstyle{definition}
\newtheorem{remark}[theorem]{Remark}
\newtheorem{definition}[theorem]{Definition}

\numberwithin{theorem}{section}
\numberwithin{equation}{section}
\numberwithin{figure}{section}

\newcommand{\<}{\begin{equation}}
\renewcommand{\>}{\end{equation}}

\newcommand{\tssum}{{\textstyle\sum}}

\newcommand{\expp}[1]{\exp\left( #1 \right)}

\newcommand{\Dir}[2]{\mathfrak{D}[#1;#2]}
\newcommand{\Mel}[2]{\mathcal{M}\left[#1;#2\right]}
\newcommand{\MelZero}[2]{\mathcal{M}^0\left[#1;#2\right]}
\newcommand{\MelInv}[2]{\mathcal{M}^{-1}[#1;#2]}

\newcommand{\pth}[1]{\left(#1\right)}
\newcommand{\ABS}[1]{\left|#1\right|}
\renewcommand{\O}[1]{O\left(#1\right)}

\newcommand{\ldint}[1]{\dint_{#1-i\infty}^{#1+i\infty}}
\newcommand{\vdint}[1]{\dint_{\left<#1\right>}}

\DeclareMathOperator{\Res}{Res}
\DeclareMathOperator{\logsc}{logsc}
\DeclareMathOperator{\sgn}{sgn}
\DeclareMathOperator{\expsc}{expsc}

\newcommand{\rt}[1]{\sqrt{#1}}
\let\mod\undefined
\newcommand{\mod}[1]{\,(#1)}

\newcommand{\simstar}{\raisebox{-0.4ex}{ $\overset{\raisebox{-0.2ex}{\scalebox{0.7}{$*$}}}{\sim}$ }}

\renewcommand{\;}{\hspace{0.04em}} 
\renewcommand{\:}{\hspace{0.08em}} 
\renewcommand{\.}{\hspace{-0.08em}}
\newcommand{\edot}{\hspace{-0.10em}\cdot\hspace{-0.16em}} 

\newcommand{\Qdel}{\vartheta} 
\newcommand{\ellexp}{\fr{1-\xs}{1-\xQ}} 

\newcommand{\PsiD}[1]{\Psi_{(#1)}}
\newcommand{\Psid}[1]{\Psi^{(#1)}}

\newcommand{\logX}{\log X}
\newcommand{\logx}{\log x}

\newcommand{\less}{\ll}
\newcommand{\great}{\gg}
\newcommand{\lf}{\left}
\newcommand{\lh}{\left}
\newcommand{\rh}{\right}

\newcommand{\signs}{\{0,\pm1\}}

\newcommand{\spnf}{(p(n,f))_\nn}

\newcommand{\spnmu}{(p(n,\xm))_\nn}
\newcommand{\spnxl}{(p(n,\xl))_\nn}
\newcommand{\sdpeqs}{\eqref{eq:sdp} and \eqref{eq:sdp*}}
\newcommand{\sapprox}{{\tiny\hspace{0.06em}\raisebox{0.18em}{$\approx$}}}

\usepackage{adjustbox}

\newcommand{\cc}{\mathbb{C}}
\newcommand{\nn}{\mathbb{N}}
\newcommand{\pp}{\mathbb{P}}

\newcommand{\fa}{\mathfrak{a}}
\newcommand{\fb}{\mathfrak{b}}
\newcommand{\fc}{\mathfrak{c}}
\newcommand{\fz}{\mathfrak{z}}

\newcommand{\fS}{\mathfrak{S}}

\newcommand{\xa}{\alpha}
\newcommand{\xd}{\delta}
\newcommand{\xe}{\epsilon}
\newcommand{\xg}{\gamma}
\newcommand{\xh}{\eta}
\newcommand{\xk}{\kappa}
\newcommand{\xl}{\lambda}
\newcommand{\xm}{\mu}

\newcommand{\xr}{\rho}
\newcommand{\xs}{\sigma}

\newcommand{\xvq}{\vartheta}
\newcommand{\vphi}{\varphi}
\newcommand{\xz}{\zeta}

\newcommand{\xG}{\Gamma}

\newcommand{\xQ}{\Theta}
\newcommand{\xO}{\Omega}

\newcommand{\cE}{\mathcal{E}}
\newcommand{\cM}{\mathcal{M}}
\newcommand{\cR}{\mathcal{R}}
\newcommand{\sW}{\mathscr{W}}

\usepackage{fancyhdr}
\usepackage{tikz,caption}
\usetikzlibrary{calc,decorations.markings}

\def\centerarc[#1](#2)(#3:#4:#5)
    { \draw[#1] ($(#2)+({#5*cos(#3)},{#5*sin(#3)})$) arc (#3:#4:#5); }

\newcommand{\addappendix}{%
  \section*{\appendixname}
  \counterwithin*{figure}{section}
  \stepcounter{section}
  \renewcommand{\thesection}{A}
  \renewcommand{\thefigure}{\thesection.\arabic{figure}}
}

\usepackage[foot]{amsaddr}
\makeatletter
\renewcommand{\email}[2][]{%
  \ifx\emails\@empty\relax\else{\g@addto@macro\emails{,\space}}\fi%
  \@ifnotempty{#1}{\g@addto@macro\emails{\textrm{(#1)}\space}}%
  \g@addto@macro\emails{#2}%
}
\makeatother

\usepackage{setspace}
\usepackage[initials,nobysame]{amsrefs}
\AtBeginDocument{%
   \def\MR#1{}
}

\makeatletter
\def\section{%
    \@startsection{section}{1}%
    \z@{.7\linespacing\@plus\linespacing}{.5\linespacing}%
    {\normalfont\large\bfseries}%
}
\makeatother

\makeatletter
\def\@seccntformat#1{%
  \protect\textup{\protect\@secnumfont
    \csname the#1\endcsname
\space\space
  }%
}
\makeatother

\begin{document}

\title{Biasymptotics of the M\"obius- and Liouville-signed partition numbers} 
\author[T.\ Daniels]{Taylor Daniels}
\address{Dept.\ of Mathematics, Purdue Univ., 150 N Univserity St, W Lafayette, IN 47907}
\email{daniel84@purdue.edu}
\subjclass[2020]{Primary: 11P55, 11P82, 11M26. \\ \indent \textit{Keywords and phrases}: partitions, M\"obius function, Liouville function.}
\begin{abstract}
    For $n \in \mathbb{N}$ let $\Pi[n]$ denote the set of partitions of $n$, i.e., the set of positive integer tuples $(x_1,x_2,\ldots,x_k)$ such that \keep{$x_1 \geq x_2 \geq \cdots \geq x_k$} and \keep{$x_1 + x_2 + \cdots + x_k = n$}. 
    Fixing $f:\mathbb{N}\to\{0,\pm 1\}$, for $\pi = (x_1,x_2,\ldots,x_k) \in \Pi[n]$ let $f(\pi) := f(x_1)f(x_2)\cdots f(x_k)$. In this way we define the {signed partition numbers} \[ p(n,f) = \sum_{\pi\in\Pi[n]} f(\pi). \]
    
    Building on the author's previous work on the quantities $p(n,\mu)$ and $p(n,\lambda)$, where $\mu$ and $\lambda$ are the M\"obius and Liouville functions of prime number theory, respectively, on assumptions about the zeros of the Riemann zeta function we establish an alternation of the terms $p(n,\mu)$ between two asymptotic behaviors as $n\to\infty$. Similar results for the quantities $p(n,\lambda)$ are established. However, it is also demonstrated that if the Riemann Hypothesis (RH) holds, then it is possible that the quantities $p(n,\lambda)$ maintain a single asymptotic behavior as $n\to\infty$. In particular, this stable asymptotic behavior occurs if, in addition to RH, it holds that all zeros of $\zeta(s)$ in the critical strip $\{0 < \Re(s) < 1\}$ are simple and the residues of $1/\zeta(s)$ at these zeros are not too large. 
    
    To formally describe these stable and alternating behaviors, the notions of asymptotic and biasymptotic sequences are introduced using a modification of the real logarithm.
\end{abstract}
\maketitle

\section{Introduction}
\label{sec:intro}

For $n \in \nn$ let $\Pi[n]$ denote the set of \emph{partitions} of $n$, i.e., the set of positive integer tuples $(x_1,x_2,\ldots,x_k)$ such that \keep{$x_1 \geq x_2 \geq \ldots \geq x_k$} and \keep{$x_1 + x_2 + \cdots + x_k = n$}. 
For fixed $f:\nn\to\signs$, for $n \in \nn$ and $\pi=(x_1,x_2,\ldots,x_k) \in \Pi[n]$ let
    \<
        \label{eq:fpi}
        f(\pi) := f(x_1)f(x_2)\cdots f(x_k).
    \>
With this we define for $n\in\nn$ the \emph{signed partition numbers} $p(n,f)$ via
    \<
        \label{eq:pnfDefin}
        p(n,f) = \sum_{\pi\in\Pi[n]} f(\pi).  
    \>

If $n \in \nn$ has prime factorization $p_1^{a_1}p_2^{a_2} \cdots p_r^{a_r}$ with distinct primes $p_i$ and all $a_i \geq 1$, then the M\"obius $\xm$ and Liouville $\xl$ functions are defined via
    \[
        \xl(n) = (-1)^{a_1 + \cdots + a_r} \qquad \text{and} \qquad \xm(n) = 
            \begin{cases}
                (-1)^r & \text{if all $a_i = 1$,} \\
                0 & \text{otherwise.}
            \end{cases} 
    \]
In \cite{daniels2023mobius}, the following asymptotic results on the sequences $\spnmu$ and $\spnxl$ are established: For all $\xe > 0$, as $n \to \infty$ one has
    \begin{align}
        \label{eq:bigOh}
        p(n,\mu) = O\big(e^{(1+\xe)\sqrt{n}}\big) \qqand p(n,\xl) = O\big(e^{(1+\xe)\sqrt{\xz(2)n}}\big),
    \end{align}
where $\xz(2) = \pi^2/6$. In addition, for $n=2k$ with $k\in\nn$, as $k\to\infty$ one has
    \<
        \label{eq:logpn}
        \log p(2k,\mu) \sim \sqrt{2k} \qqand \log p(2k,\xl) \sim \sqrt{\zeta(2)2k},
    \>
where the relation $a_m \sim b_m$ indicates that $\lim_{m\to\infty} a_m/b_m = 1$.

This paper builds on the results of \cite{daniels2023mobius} to consider, in essence, the extent to which the relations \eqref{eq:logpn} hold when the input $n$ is odd. The statements of our main results require additional vocabulary, so for the sake of exposition we first state approximate versions of these results and then formalize them in section \ref{sec:Biasymp}. Here we use $\approx$ to mean ``behaves like, within reasonable error.''

Let $\xz(s)$ denote the Riemann zeta function and let
    \<
        \xQ := \sup\{\Re(\rho) : \xz(\rho) = 0\}.
    \>
It is well known that $1/2 \leq \xQ \leq 1$; the assertion that $\xQ=1/2$ is the Riemann Hypothesis. Our three primary results extending the relations \eqref{eq:logpn} are then roughly described in the following approximate statements:
\emph{
\begin{enumerate}[label={\normalfont(\arabic*\sapprox)}]
    \item \label{it:mu} If $\xQ<1$, then for all $L \in \nn$ there exist arbitrarily large $N,N_* \in \nn$ such that
        \begin{align*}
            \log p(n,\xm) &\approx \rt{n} \qquad \text{for $N \leq n \leq N+L-1$}, \\
            \log\!\big((-1)^np(n,\mu)\big) &\approx \rt{n} \qquad \text{for $N_* \leq n \leq N_*+L-1$.}
        \end{align*}
    \item \label{it:xl} If $\xQ<1$ but $\xQ>1/2$ as well, then for all $L \in \nn$ there exist arbitrarily large $N,N_* \in \nn$ such that
        \begin{align*}
            \log p(n,\xl) &\approx \rt{\xz(2)n} \qquad \text{for $N \leq n \leq N+L-1$}, \\
            \log\!\big((-1)^np(n,\xl)\big) &\approx \rt{\zeta(2)n} \qquad \text{for $N_* \leq n \leq N_*+L-1$.}
        \end{align*}
    \item \label{it:asm} If $\xQ=1/2$, then the long-term behavior of $\spnxl$ is uncertain. In fact, it is possible that for some $N_0 \in \nn$ one has
        \<
            \label{eq:3axp}
            \log\!\big((-1)^n p(n,\xl)\big) \approx \rt{\xz(2)n} \qquad \text{for \emph{all} $n>N_0$.}    
        \>
\end{enumerate}}

Considering items \ref{it:mu}, \ref{it:xl}, and \ref{it:asm} above we see, informally speaking, that on some strings $N,N+1,\ldots,N+L-1$ the relations \eqref{eq:logpn} can be ``directly'' extended to include odd $n$, but that on different strings $N_*, N_*+1,\ldots, N_*+L_*-1$ the extension requires the inclusion of a factor $(-1)^{n}$.

\subsection*{Acknowledgements.}
The author would like to thank Trevor Wooley for financial support and additionally specially thank Trevor Wooley and Ben McReynolds for their numerous, invaluable suggestions and feedback in preparing this work. Thanks are also extended to Nicolas Robles and Alexandru Zaharescu for helpful conversations in formulating this paper.

\section{Asymptotic and biasymptotic sequences}
\label{sec:Biasymp}

To motivate precise versions of statements \ref{it:mu}, \ref{it:xl}, and \ref{it:asm} we recall three classical results on certain sequences $(p(n,1_A))_\nn$, where $1_A$ is the indicator function for $A \subset \nn$. For brevity let $p(n,A):=p(n,1_A)$. The most well known of these three is the relation 
    \<
        \label{eq:LogPn1sim}
        \log p(n,\nn) \sim \xk \rt{n}, 
    \>
where
    \[
        \xk := \pi \sqrt{2/3},    
    \]
which is an immediate corollary of Hardy and Ramanujan's main results in \cite{hardy1918asymptotic}.

To state the second result we recall that $A \subset \nn$ is said to have \emph{density} $\xd_A$ if the ratio $|A \cap \{1,2,\ldots,N\}|N^{-1}$ tends to $\xd_A$ as $N\to\infty$. It is evident that $\xd_\nn = 1$, and it is well known that the set of squarefree numbers has density $6/\pi^2 \approx 0.601$. The prime number theorem implies that $\xd_\pp = 0$. A remarkable theorem of Erd\H{o}s \cite{erdos1942elementary} states that: if $A \subset \nn$ and $\mathrm{gcd}(A)=1$, then $A$ has density $\xd_A > 0$ if and only if
    \<
        \label{eq:Erdossim}
        \log p(n,A) \sim \xk\rt{\xd_A n}.
    \>
We remark that the proof of the ``only if'' statement in Erd\H{o}s' theorem is omitted in \cite{erdos1942elementary} since the proof of a similar result is given there. Complete proofs of both statements are given in \cite{nathanson2008elementary}*{Ch.\ 16}. 

The third result, due to Roth and Szekeres \cite{roth1954some}, states that as $n\to\infty$ one has
	\<
		\label{eq:LogPrimessim}
		\log p(n,\pp) \sim \xk \rt{\fr{2n}{\log n}},
	\>
where $\pp$ is the set of prime numbers. Finally we note that we may rewrite relations \eqref{eq:bigOh} and \eqref{eq:logpn} and statements \ref{it:xl} and \ref{it:asm} with $\rt{\xz(2)} = \tf12\xk = \tf12\pi\rt{2/3}$.

Computing $p(n,\xm)$ and $p(n,\xl)$ for $n \leq 10^5$ we notice the following: 
\begin{enumerate}[label=(\roman*)]
    \item Except for very small $n$, say $n \leq 50$, we have
        \<
            \label{eq:empirical}
            p(n,\xm) \approx (-1)^n e^{\rt{n}} \qqand p(n,\xl) \approx (-1)^n e^{\fr12\xk\rt{n}}.
        \>
    \item \label{it:zero} here are some $n$ for which $p(n,\xm)$ or $p(n,\xl)$ are 0, but all such $n$ are $\leq 25$.
\end{enumerate}
Thus, because these $p(n,f)$ may be positive, negative, or zero we cannot simply speak of $\log p(n,f)$, and instead employ the following modification of the logarithm. 

For real $x$ we define the \emph{logarithmic scale} function $\logsc(x)$ via
	\<
		\logsc(x) = \sgn(x)\log(|x|+1),    
	\>
where $\sgn(0)=0$ and $\sgn(x)=x/|x|$ otherwise.
One may easily check that $\logsc x$ is continuously differentiable and invertible with inverse 
    \[
        \expsc(x) := \sgn(x)(\exp|x|-1),
    \] 
and additionally that $\logsc{x}\sim\logx$ as $x\to\infty$. Using $\logsc(x)$ we may characterize sequences that behave like $(-1)^n \expp{x_n+o(x_n)}$ for some sequence $(x_n)$ that grows to $+\infty$, as seen in the following lemma. 

\begin{lemma}
    \label{lem:logsc}
    Let $(a_n)$ and $(x_n)$ be real sequences with $x_n > 0$ for all $n$, and suppose that $x_n \to \infty$ as $n\to\infty$. The relation $(-1)^n \logsc a_n \sim x_n$ holds if and only if
        \<
            \label{eq:LogscOnlyIf}
            a_n = (-1)^n \lf( e^{x_n+y_n}-1 \rh)
        \>
    for some sequence $(y_n)$ satisfying the relation $y_n = o(x_n)$. 
\end{lemma}

\begin{proof}
    Let $b_n := \logsc{a_n}$. If $(-1)^n b_n \sim x_n$, then there exists a sequence $(y_n)$ with $y_n = o(x_n)$ for which $b_n = (-1)^n x_n + y_n$ for all $n$. By restricting to sufficiently large values of $n$ if necessary, there is no loss of generality in assuming that $y_n < x_n$ for all $n$. It follows that $\sgn(b_n) = (-1)^n$, that $|b_n| = (-1)^n b_n = x_n + (-1)^n y_n$, and that
        \[
            a_n = \expsc{b_n} = (-1)^n \pth{ e^{x_n+(-1)^ny_n}-1 }.    
        \]
    Formula \eqref{eq:LogscOnlyIf} then follows by simply relabeling $(-1)^n y_n$ as $y_n$. The reverse implication in the lemma is immediate.
\end{proof}

\emph{Strings} are tuples $(N,N+1,\ldots,N+L-1)$ where $N,L \in \nn$; the string's \emph{length} is $L$. Length $L=1$ and length $L>1$ strings are trivial and nontrivial, respectively.

Although computations for small $n$ suggest that $p(n,\xm) \approx (-1)^n \exp(\rt{n})$ as $n\to\infty$, our approximate result \ref{it:mu} above states that if all zeros of $\zeta(s)$ have real part $\leq 1-\xe$ for some $\xe>0$, then there exist arbitrarily long strings on which $p(n,\xm) \approx \exp(\rt{n})$ and arbitrarily long strings on which $p(n,\xm) \approx (-1)^n \exp(\rt{n})$, and that $p(n,\xm)$ alternates between these two approximate behaviors indefinitely. To formally describe this alternation, for sequences $(x_n)$ and $(y_n)$ we say that
    \<
        x_n \simstar y_n
    \>
if: for all $\xe > 0$ there exist arbitrarily long strings on which $|x_n/y_n-1|<\xe$. It is equivalent to say that $x_n\simstar y_n$ if, for all $\xe>0$ and $L\in\nn$, there exist arbitrarily large $N$ such that
    \<
        \lh|\nfr{x_n}{y_n} - 1\rh|<\xe \qquad \text{for $N \leq n \leq N+L-1$.}
    \>

\begin{definition}
    \label{def:Biasymp}
    For fixed $\xa\in(0,1]$, we say a sequence $(p(n,f))_\nn$ is \emph{biasymptotic with factor $\xa$} if both
        \begin{align}
            \logsc p(n,f) \simstar \xa\xk\rt{n} \qqand 
            (-1)^n \logsc p(n,f) \simstar \xa\xk\rt{n},
        \end{align}
    where $\xk=\pi\rt{2/3}$ is the ``asymptotic factor'' for $\log p(n,\nn)$ in \eqref{eq:LogPn1sim}.
\end{definition}

We note that, by Lemma \ref{lem:logsc}, the latter relation here is only possible if on arbitrarily long strings one has $p(n,f) = (-1)^n e^{\xa\xk\rt{n} + y_n}$ for some terms $y_n$ dominated by $\rt{n}$. 
Recalling that $\xQ := \sup\{\Re(\rho) : \xz(\rho) = 0\}$, we now formalize statements \ref{it:mu} and \ref{it:xl}.

\begin{theorem}[cf.\ \ref{it:mu}]
    \label{thm:muBiasymp}
    If $\xQ < 1$, then $\spnmu$ is biasymptotic with factor $\xk^{-1}$.
\end{theorem}

\begin{theorem}[cf.\ \ref{it:xl}]
    \label{thm:nonRHBiasymp}
    If $\fr12 < \xQ < 1$, then $\spnxl$ is biasymptotic with factor $\fr12$.
\end{theorem}

As mentioned in statement \ref{it:asm}, the behavior of $\spnxl$ on assumption of the Riemann Hypothesis (RH) is less certain than that of $\spnmu$. In particular, additional information regarding the nontrivial zeros of $\xz(s)$ is required to draw conclusions. The Simplicity Hypothesis (SH) asserts that zeros $\rho$ of $\xz(s)$ in the strip $\{0 < \Re(s) < 1\}$ are simple, and we have the following other scenario in which $\spnxl$ is biasymptotic.

\begin{theorem}
    \label{thm:nonsimple}
    Suppose that RH holds but SH does not, i.e., that some zero $\tf12+i\xg$ of $\xz(s)$ has order $m>1$. Then $\spnxl$ is biasymptotic with factor $\fr12$.
\end{theorem}

For the sake of consistency between the formal versions of \ref{it:mu}, \ref{it:xl}, and \ref{it:asm}, for fixed $\xa\in(0,1]$ we say a sequence $\spnf$ is \emph{asymptotic with factor $\xa$} if one has 
    \<
        \label{eq:asymp}
        \logsc p(n,f) \sim \xa\xk\rt{n} \qquad (\xk = \pi\rt{2/3}).
    \>
Recalling that $\xm^2$ is the indicator function for the set of squarefree integers, it is immediate from \eqref{eq:LogPn1sim} and \eqref{eq:Erdossim} that $(p(n,1))_\nn$ and $(p(n,\mu^2))_\nn$ are asymptotic with factors $1$ and $\rt{6}/\pi \approx 0.7797$, respectively. In light of relation \eqref{eq:LogPrimessim}, the sequence $(p(n,\pp))_\nn$ is not described by this definition because we require that $\xa$ be constant. On the other hand, one may easily use a different ``normalization'' in \eqref{eq:asymp}, e.g., specify that 
    \[
        \logsc p(n,f) \sim \xa\xk \rt{\fr{2n}{\log n}}.
    \]

If RH and SH hold and there exists some $C > 0$ such that 
    \<
        |1/\xz'(\rho)| \leq C|\rho|
    \>
for all zeros $\rho = 1/2+i\xg$ of $\xz(s)$ on the critical line, we say that the Effective Simplicity Hypothesis (ES) holds with constant $C$. At last we formalize statement \ref{it:asm}.

\begin{theorem}[cf.\ \ref{it:asm}]
    \label{thm:RAsymp}
    Suppose that RH, SH, and ES hold, the latter with constant $C > 0$. There exists a constant $\fc > 0$ such that: if the ES constant $C < \fc$, then
        \<
            (-1)^n \logsc p(n,\xl) \sim \tf12 \xk\rt{n},
        \>
    namely $((-1)^n p(n,\xl))_\nn$ is asymptotic with factor $\fr12$. Moreover, it holds that 
    {\binoppenalty=10000\relpenalty=10000 $\fc > 10^{17881}$}.
\end{theorem}

Finally we plot $\logsc p(n,\xm)$ and $\logsc p(n,\xl)$ for some small $n$ in Figures \ref{fig:lscpn250} and \ref{fig:lscpn100k}.

{\begin{figure}[ht]
    \centering
    \includegraphics[width=0.83\textwidth,keepaspectratio]{"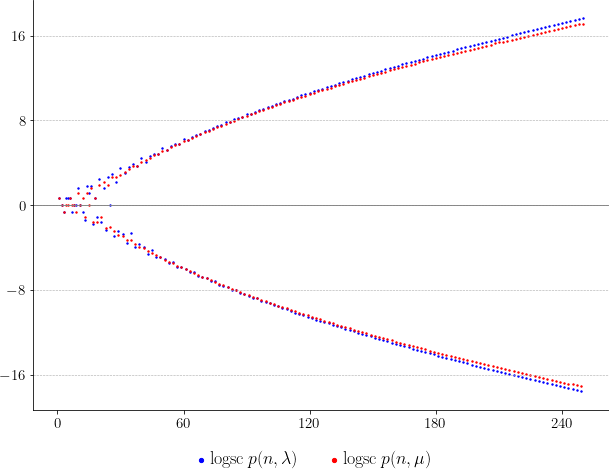"}
    \caption{Plot of $\logsc p(n,\xm)$ and $\logsc p(n,\xl)$ for $n \leq 250$.}
    \label{fig:lscpn250}
\end{figure}}
    
{\begin{figure}[ht]
    \centering
    \includegraphics[width=0.85\textwidth,keepaspectratio]{"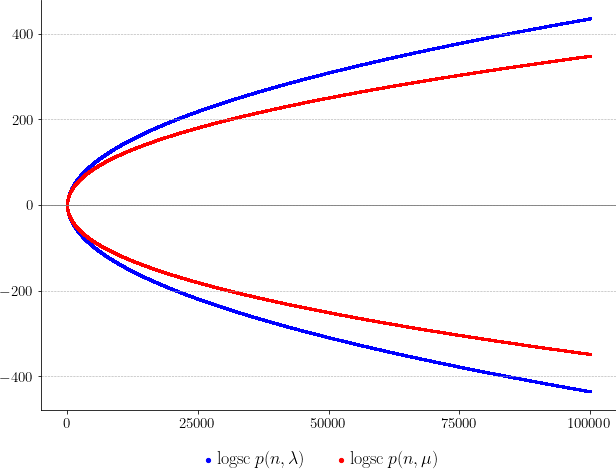"}
    \caption{Plot of $\logsc p(n,\xm)$ and $\logsc p(n,\xl)$ for $n \leq 100,000$.}
    \label{fig:lscpn100k}
\end{figure}}

\section{Notation and structure}
\label{sec:prelim}

Expressions $O(g(x))$ with $g(x) > 0$ indicate some quantity bounded by a constant multiple of $g(x)$, and $f(x) \less g(x)$ indicates that $f(x) = O(g(x))$ for all sufficiently large $x$. The relation $f(x) \asymp g(x)$ holds if $f(x) \less g(x)$ and $g(x) \less f(x)$. Quantities of the form $X_0(A)$ indicate some $X_0 > 0$ depending on $A$. Statements involving $\xe$ are understood to hold for all sufficiently small $\xe > 0$ unless noted otherwise, and constants' dependencies on $\xe$ are generally suppressed.  

For real $\xa$ let $e(\xa)=\exp(2\pi i\xa)$. Expressions $n \equiv a$ (mod $q$) are abbreviated as $n \equiv a \mod{q}$. For nonnegative integer $j$ we let $\PsiD{j}(z) = (z\partial_z)^j \Psi(z)$. Functions $\Phi(z,f)$, $\Psi(z,f)$, etc. are simply written as $\Phi(z)$, $\Psi(z)$, etc. when the $f$ used is clear or when the discussions apply to general $f:\nn\to\signs$.

For $\rho, \rho_* \in (0,1)$ it is convenient to write
    \<
        \label{eq:X}
        X = \fr{1}{\log(1/\rho)} \qqand X_* = \fr{1}{\log(1/\rho_*)}
    \>
so that
    \<
        \label{eq:rho}
        \rho = e^{-1/X} \qqand \xr_* = e^{-1/X_*},
    \>
and this convention is maintained throughout this paper.

For complex $s$ we write $s=\xs+it$, and $\xs$ and $t$ always represent the real and imaginary parts of some $s\in\cc$, respectively. Line integrals $\int_{\xs-i\infty}^{\xs+i\infty}$ are abbreviated as $\int_{\lh<\xs\rh>}$, and integrals $\frac{1}{2\pi i} \int$ are abbreviated using dashed integrals $\dint$.

\subsection*{Structure and Overview.}
In section \ref{sec:recall} we begin our analyses by recalling the relevant asymptotic relations and formulae for $p(n,\xm)$ and $p(n,\xl)$ established in \cite{daniels2023mobius}. We then reframe these formulae into  forms that allow us to identify the oscillatory quantities responsible for the (bi-)asymptotic behaviors. In particular, we find that both $p(n,\xm)$ and $p(n,\xl)$ are expressible as
    \<
        \label{eq:pnfTemplate}
        p(n,f) = e^{\phi(n)} \lf[ e^{\phi_*(n)} + (-1)^n e^{-\phi_*(n)} + O(n^{-1/5}e^{|\phi_*(n)|}) \rh]   
    \>
for some $\phi(n) \asymp \rt{n}$ and some slowly oscillating $\phi_*(n)$ that is $o(\rt{n})$ (both different for each $f$). Thus, the majority of our analyses concern the functions $\phi_*(n,\xm)$ and $\phi_*(n,\xl)$.

In section \ref{sec:OmegaPrelim} we recall the $\xO$ asymptotic relation and establish a template for our biasymptotic results using \eqref{eq:pnfTemplate}. In addition we collect some conditional bounds on $\xz(s)$.

Sections \ref{sec:Mpp*} and \ref{sec:MBiasymp} specifically focus on $\phi(n,\mu)$ and $\phi_*(n,\xm)$ and the proof of Theorem \ref{thm:muBiasymp}, and after this sections \ref{sec:Lpp*} and \ref{sec:LBiasymp} do the same for $\phi(n,\xl)$ and $\phi_*(n,\xl)$ but with the caveat that the Riemann Hypothesis (RH) does \emph{not} hold. Sections \ref{sec:Riemann}--\ref{sec:RiemannExact} then specially treat $\phi_*(n,\xl)$ and $p(n,\xl)$ on assumption of RH and discuss other assumptions concerning the nontrivial zeros of $\xz(s)$.

Finally, section \ref{sec:explicit} revisits $\phi_*(n,\xm)$ on assumption of RH and makes rough estimates on when biasymptotic behavior of $\spnmu$ might be first explicitly observed (assuming such behavior is ever adopted at all).

\section{\tops{A review of results on $p(x,\xm)$ and $p(n,\xl)$}}
\label{sec:recall}

Our analyses of the potential (bi-)asymptotic behaviors of $\spnmu$ and $\spnxl$ require detailed analyses of the generating functions of these sequences and how they affect the asymptotic formulae for $p(n,\xm)$ and $p(n,\xl)$ established in \cite{daniels2023mobius}*{Thms.\ 1.2--1.5}. We briefly recall these formulae and the necessary tools from sections 1 and 2 of \cite{daniels2023mobius}. 

The functions $\Phi(z) = \Phi(z,f)$ and $\Psi(z)=\Psi(z,f)$ associated to the sequence $\spnf$ are defined for $|z|<1$ via
    \<
        \label{eq:Phi}
    	\Phi(z) = \prod_{n=1}^\infty \pfrac{1}{1-f(n)z^n} \qqand \Psi(z) = \sum_{k,n=1}^\infty \fr{f^k(n)}{k} z^{nk}
    \>
so that 
    \<
        \Phi(z) = \exp \Psi(z).
    \>
By comparison to the functions $\Phi(z,1)$ and $\Psi(z,1)$ one sees that $\Phi(z,f)$ and $\Psi(z,f)$ are analytic for $|z|<1$, so by Cauchy's theorem one has
	\<
		\label{eq:pnfInt}
		p(n,f) = \fr{1}{2\pi i}\int_{|z|=\rho} \Phi(z)z^{-n-1} \,dz = \rho^{-n} \int_0^1 \Phi(\rho e(\xa)) e(-n\xa) \,d\xa
	\>
for all $0<\rho<1$. The integrals \eqref{eq:pnfInt} provide a natural definition for the quantities $p(x,f)$ for a general $x>0$ by simply using $x$ in place of $n$ therein.

To obtain asymptotic results on $p(x,f)$ for a given $x>0$ using \eqref{eq:pnfInt}, we select $\xr=\xr(x)$ and $\xr_*=\xr_*(x)$ so as to minimize the quantities $\Psi(\rho)-x\log\rho$ and $\Psi(-\rho_*)-x\log\rho_*$. The minima of these expressions occur when $\xr$ and $\xr_*$ satisfy the \emph{saddle-point equations}
    \begin{align}
        \label{eq:sdp}
        \rho \Psi'(\rho) &= x, \tag{sp} \\
        \label{eq:sdp*}
        (-\rho_*)\Psi'(-\rho_*) &= x, \tag{sp$*$}
    \end{align}
where $(-\xr_*)\Psi'(-\xr_*)$ is $z\Psi'(z)$ evaluated at $z=-\rho_*$. If $f:\nn\to\signs$ takes positive and negative values, one should not expect that equations \eqref{eq:sdp} and \eqref{eq:sdp*} have solutions in $(0,1)$ for all $x>0$. However, when $f=\xm$ or $f=\xl$ and $x$ is sufficiently large, both of these equations have \emph{unique} solutions $\xr,\xr_*\in(0,1)$, as established in \cite{daniels2023mobius}*{Lem.\ 8.1} and an analogous argument regarding $\Psi(z,\xl)$.

When $x$ is sufficiently large and $r\in(0,1)$ is a solution to one of the equations \eqref{eq:sdp} and \eqref{eq:sdp*} with $\Psi(z)=\Psi(z,\xm)$, we have the following relations between $x$, $r$, and certain derivatives of $\Psi(\pm r)$. To preempt any confusion, we note that $(-1)^{j}\PsiD{j}(-r)$ and $\PsiD{j}(-r)$ are uniquely valued independently of their interpretations vis-\`a-vis the chain-rule.

\begin{theorem}
    \label{M:thm:Relations}
    Let $A > 0$ and let $\Psi(z)=\Psi(z,\xm)$. For all sufficiently large $x$, if $r = r(x)$ satisfies at least one of the equations \eqref{eq:sdp} and \eqref{eq:sdp*}, then
        \begin{align*}
            - x \log r &= \tf12 \rt{x} \lf[ 1+O((\logx)^{-A}) \rh], \\
            \Psi(\pm r) &= \tf12 \rt{x} \lf[ 1+O((\logx)^{-A}) \rh].
        \end{align*}
    In general, for these $x$ and $r$ and for $j \geq 1$ one has
        \begin{align*}
            (\pm 1)^j\Psid{j}(\pm r) &= 2^{j-1}j!\: x^{\fr{j+1}{2}} \lf[ 1+O_j((\logx)^{-A}) \rh], \\
            \PsiD{j}(\pm r) &= 2^{j-1}j!\: x^{\fr{j+1}2} \lf[ 1+O_j((\logx)^{-A}) \rh].
        \end{align*}
\end{theorem}

We now recall the asymptotic formula for $p(x,\mu)$ in terms of $\Phi(z,\xm)$ and $\Psi(z,\xm)$.

\begin{theorem}
    \label{M:thm:asymp}
    Let $\Phi(z)=\Phi(z,\xm)$ and $\Psi(z)=\Psi(z,\xm)$, and for sufficiently large $x$ let $\rho$ and $\rho_*$ denote the unique solutions to \eqref{eq:sdp} and \eqref{eq:sdp*}, respectively. As $x \to \infty$, one has
        \<
            \label{M:eq:pxmuForm}
            p(x,\xm) = \fr{\rho^{-x}\Phi(\rho)}{\rt{2\pi \PsiD{2}(\rho)}} \big[ 1+O(x^{-1/5}) \big] + (-1)^{-x} \fr{\rho_*^{-x}\Phi(-\rho_*)}{\rt{2\pi \PsiD{2}(-\rho_*)}} \big[ 1+O(x^{-1/5}) \big],
        \>
    where $(-1)^{-x} = \exp(-\pi ix)$.
\end{theorem}

For the case of $\Phi(z,\xl)$ and $\Psi(z,\xl)$ we have results highly similar to those above. Due to its ubiquity in our discussions, we let
    \<
        \fz = \xz(2) = \tf{\pi^2}{6},
    \>
and we recall that
    \[
        \xk = 2\sqrt{\xz(2)} = \pi \rt{2/3}.    
    \]
\begin{theorem}
	\label{L:thm:Relations}
	Let $A>0$ and let $\Psi(z)=\Psi(z,\xl)$. For all sufficiently large $x$, if $r = r(x)$ satisfies at least one of the equations \eqref{eq:sdp} and \eqref{eq:sdp*}, then
		\begin{align*}
			- x \log r &= \tf14 \xk\rt{n} \lf[ 1+O((\logx)^{-A}) \rh], \\
            \Psi(\pm r) &= \tf14 \xk\rt{n} \lf[ 1+O((\logx)^{-A}) \rh].
        \end{align*}
    In general, for these $x$ and $r$ and for $j \geq 1$ one has
        \begin{align*}
            (\pm1)^{j}\Psid{j}(\pm r) &=  j!\: (\xk/\fz)^{j-1} x^{\fr{j+1}2} \lf[ 1+O_j((\logx)^{-A}) \rh], \\
            \PsiD{j}(\pm r) &= j!\: (\xk/\fz)^{j-1} x^{\fr{j+1}2} \lf[ 1+O_j((\logx)^{-A}) \rh],
        \end{align*}
\end{theorem}

\begin{theorem}
    \label{L:thm:pnlAsymp}
    Let $\Phi(z)=\Phi(z,\xl)$ and $\Psi(z)=\Psi(z,\xl)$, and for sufficiently large $x$ let $\rho$ and $\rho_*$ denote the unique solutions to \eqref{eq:sdp} and \eqref{eq:sdp*}, respectively. As $x\to\infty$ one has
        \<
            \label{L:eq:pxxlForm}
            p(x,\xl) = \fr{\rho^{-x}\Phi(\rho)}{\rt{2\pi \PsiD{2}(\rho)}} \big[ 1+O(x^{-1/5}) \big] + (-1)^{-x} \fr{\rho_*^{-x}\Phi(-\rho_*)}{\rt{2\pi \PsiD{2}(-\rho_*)}} \big[ 1+O(x^{-1/5}) \big],
        \>
    where $(-1)^{-x} = \exp(-\pi ix)$.
\end{theorem}

\begin{remark}
    For brevity, the definition $(-1)^{-x}=e^{-\pi ix}$ is understood going forward.
\end{remark}

\subsection{\tops{$\phi(x)$ and $\phi_*(x)$.}}
\label{sec:pp*}

As seen in formula \eqref{M:eq:pxmuForm}, whether or not the quantities $p(n,\xm)$ alternate sign depends on which of the two main terms in the formula is larger. To put \eqref{M:eq:pxmuForm} into a form better-suited to our comparisons, we first rewrite \eqref{M:eq:pxmuForm} as
    \<
        \label{eq:pxmupsi}
        p(x,\mu) = e^{\psi(x)} \big[ 1+O(x^{-\nfr15}) \big] + (-1)^{-x} e^{\psi_*(x)} \big[ 1+O(x^{-\nfr15}) \big],
    \>
where $\psi(x) = \psi(x,\mu)$ and $\psi_*(x)=\psi_*(x,\mu)$ are given by
    \begin{align}
        \label{M:eq:psix}
        \psi(x) &= \Psi(\rho) - x \log(\rho) - \tf12 \log(2\pi\PsiD{2}(\rho)), \\
        \label{M:eq:psi*x}
        \psi_*(x) &= \Psi(-\rho_*) - x \log(\rho_*) - \tf12 \log(2\pi\PsiD{2}(-\rho_*)).
    \end{align}

Using the elementary manipulation 
    \[
        e^u + (-1)^{-x} e^v = e^{\fr12(u+v)} \pth{ e^{\fr12(u-v)} + (-1)^{-x} e^{-\fr12(u-v)} },
    \]
and bounding $e^w + e^{-w} = O(e^{|w|})$, we rewrite \eqref{eq:pxmupsi} (for sufficiently large $x$) as
    \<
        \label{M:eq:pxmuAsymp3}
        p(x,\mu) = e^{\phi(x)} \Big[ e^{\phi_*(x)} + (-1)^{-x} e^{-\phi_*(x)} + O(x^{-\nfr15}e^{|\phi_*(x)|}) \Big],
    \>
where 
    \[ 
        \phi(x) := \tf12(\psi(x)+\psi_*(x)) \qqand \phi_*(x) := \tf12(\psi(x)-\psi_*(x)).
    \] 
Explicitly, using \eqref{M:eq:psix} and \eqref{M:eq:psi*x} we have
    \begin{align}
        \label{M:eq:phix0}    
        \phi(x) &= \tf12(\Psi(\xr)+\Psi(-\xr_*)) - \tf12 x \log(-\xr\xr_*) - \tf14 \log\lf( 4\pi^2\PsiD{2}(\xr)\PsiD{2}(-\xr_*) \rh), \\
        \label{M:eq:phi*x0}
        \phi_*(x) &= \tf12(\Psi(\xr)-\Psi(-\xr_*)) - \tf12 x \log(\xr/\xr_*) - \tf14 \log\lf( \PsiD{2}(\xr)/\PsiD{2}(-\xr_*) \rh).
    \end{align}

With the relations of Theorem \ref{M:thm:Relations} we may quickly check that
    \begin{align*}
        \phi(x) &= \sqrt{x} + O_A(\sqrt{x} (\logx)^{-A}), \\
        \phi_*(x) &= O_A(\sqrt{x} (\logx)^{-A}),
    \end{align*}
again demonstrating the bound \eqref{eq:bigOh} for $p(n,\xm)$. In addition, since $(-1)^{-x} = +1$ if $x=2k$ for some $k \in \nn$, we find (unconditionally) that
    \[
        p(2k,\mu) = e^{\rt{2k} + o(\rt{2k})} \lf[ e^{o(\rt{2k})} \rh],
    \] 
confirming the relation \eqref{eq:asymp} for $\log p(2k,\mu)$. Using the conventions \eqref{eq:rho} for $X$ and $X_*$ we rewrite \eqref{M:eq:phix0} and \eqref{M:eq:phi*x0} as
    \begin{align*}
        \phi(x) &= \fr12 \lf[ \Psi(\xr)+\Psi(-\xr_*) + \frac{x(X + X_*)}{XX_*} - \fr12 \log\lf( 4\pi^2 \PsiD{2}(\xr)\PsiD{2}(-\xr_*) \rh) \rh], \\
        \phi_*(x) &= \fr12 \lf[ \Psi(\xr) - \Psi(-\xr_*) - \frac{x(X - X_*)}{XX_*} - \fr12 \log\pfrac{\PsiD{2}(\xr)}{\PsiD{2}(-\xr_*)} \rh],
    \end{align*}
which are the specific formulae we analyze in section \ref{sec:Mpp*}.

For easy reference we summarize these manipulations in the following lemma.

\begin{lemma}
    \label{M:lem:pp*}
    Fix $A>0$, let $\Psi(z)=\Psi(z,\mu)$, and for all sufficiently large $x$ let $\xr$ and $\xr_*$ be the unique solutions to equations \eqref{eq:sdp} and \eqref{eq:sdp*}, respectively. As $x\to\infty$ one has
        \[
            p(x,\mu) = e^{\phi(x)} \Big[ e^{\phi_*(x)} + (-1)^{-x} e^{-\phi_*(x)} + O(x^{-1/5}e^{|\phi_*(x)|}) \Big],
        \]
    where $(-1)^{-x} = \exp(-\pi ix)$ and
        \begin{align}
            \label{M:eq:phix}
            \phi(x) &= \fr12 \lf[ \Psi(\xr)+\Psi(-\xr_*) + \frac{x(X + X_*)}{XX_*} - \fr12 \log\lf( 4\pi^2 \PsiD{2}(\xr)\PsiD{2}(-\xr_*) \rh) \rh], \\
            \label{M:eq:phi*x}
            \phi_*(x) &= \fr12 \lf[ \Psi(\xr) - \Psi(-\xr_*) - \frac{x(X - X_*)}{XX_*} - \fr12 \log\pfrac{\PsiD{2}(\xr)}{\PsiD{2}(-\xr_*)} \rh].
        \end{align}
    Moreover, as $x\to\infty$ one has
        \begin{align*}
            \phi(x) &= \sqrt{x} + O_A(\sqrt{x} (\logx)^{-A}), \\
            \phi_*(x) &= O_A(\sqrt{x} (\logx)^{-A}).
        \end{align*}
\end{lemma}

A similar result holds for $p(x,\xl)$ but we state it below at a more appropriate time.

\section{Preliminaries for \tops{$\xO$}--results}
\label{sec:OmegaPrelim}

\subsection{\tops{The $\xO$ and $\xO^*$ relations.}}
\label{secsub:OmegaPrelim}

In this section let $f(x)$ and $g(x)$ denote real-valued functions defined on $(0,\infty)$ with $g(x) > 0$ everywhere. We recall that $f(x)=\xO_+(g(x))$ if for some $C > 0$ there exist arbitrarily large $x$ such that 
    \<
        \label{eq:OmegaP}
        f(x) > C g(x),
    \>
and that $f(x)=\xO_-(g(x))$ if for some $C > 0$ there exist arbitrarily large $x$ such that 
    \<
        \label{eq:OmegaM}
        f(x) < -C g(x).
    \>

Our analyses of $\phi_*(x,\xm)$ and $\phi_*(x,\xl)$ require a stronger relation than that of $f(x) = \xO_{\pm}(g(x))$. Indeed, the $x$ in the relations \eqref{eq:OmegaP} and \eqref{eq:OmegaM} may not be integers, nor can we a priori guarantee that $f(y)>Cg(y)$ continues to hold for $y$ in large neighborhoods of said $x$. We thus define the following stronger relation.

\begin{definition}
    \label{def:Omega*}
    We say that $f(x) = \xO_+^*(g(x))$ if there exists $C>0$ with the following property: For all $r>0$, there exist arbitrarily large $x$ such that
        \<
            \label{eq:Omega*Plus}
            f(y) > C g(x) \qquad \text{for $y \in [x,x+r)$}.
        \>
    We similarly say that $f(x) = \xO_-^*(g(x))$ if there exists $C>0$ such that: For all $r > 0$, there exist arbitrarily large $x$ such that 
        \<
            \label{eq:Omega*Minus}
            f(y) < -C g(x) \qquad \text{for $y \in [x,x+r)$}.
        \>
\end{definition}

Definition \ref{def:Omega*} in hand, we aim to establish $\xO^*$--results for $\phi_*(x,f)$ by approximating it with a simpler quantity $\phi_0(x,f)$ for which we can establish $\xO^*$--relations. We record two elementary lemmata used repeatedly in our analyses towards this goal.

\begin{lemma}
	\label{B:lem:DerivDecay}
	Suppose that $f(x)$ is continuously differentiable on $(0,\infty)$ and that there exist constants $a$ and $A$, with $0 < a < 1$ and $A>0$, such that $f(x) = \xO_{\pm}(x^{a})$ and
		\<
            \label{B:eq:DerivDecay}
            |f'(x)|\leq \fr{A}{x^{1-a}} \qquad (x > 0).
        \>
	Then $f(x) = \xO_{\pm}^*(x^a)$.
\end{lemma}

\begin{proof}
	Let $r>0$. Because $f(x) = \xO_+(x^{a})$ there exists $C>0$ such that $f(x) > C x^a$ for arbitrarily large values of $x$. Thus, suppose that $x$ is large and $f(x) > C x^{a}$. By the mean value theorem and \eqref{B:eq:DerivDecay}, for all $y \in (x,x+r)$ there exists $\xi \in (x,y)$ such that
		\<
			\label{B:eq:fMeanValue}
			f(y) = f(x) + f'(\xi)(y-x) > C x^a - A \xi^{a-1} r.
		\>

	As $a-1 < 0$ one has $\xi^{a-1} < x^{a-1}$, whereby
		\<
			\label{B:eq:fLowerBB}
			f(y) > C x^a - A x^{a-1} r = x^{a} \pth{C - A r / x} \qquad (y \in (x,x+r)).
		\>
    The relation $f(x) = \xO_+^*(x^a)$ follows at once since $x$ may be arbitrarily large, and the relation $f(x) = \xO_-^*(x^{a})$ is proved similarly.
\end{proof}

\begin{lemma}
    \label{B:lem:Omega*Approx}
    Suppose that $f(x) = \xO_\pm^*(x^{a-\xe})$ for some $a > 0$ and that 
        \[
            g(x) = f(x) + O(x^{a-\Qdel+\xe})
        \] 
    for some $0 < \Qdel \leq a$ and all $x>0$. Then $g(x) = \xO_\pm^*(x^{a-\xe})$. 
\end{lemma}

\begin{proof}
    Let $0<\xe<\min\{a,\Qdel/2\}$ and let $C>0$ be a constant satisfying the $\xO_+^*$ condition for $f(x)$; thus, for $r>0$ there exist arbitrarily large $x$ such that
        \[
            f(y) > C x^{a-\xe} \qquad \text{for all $y \in [x,x+r)$.}    
        \]
    
    On this interval then, for some $c_\xe > 0$ we have
        \[
            g(y) > C x^{a-\xe} - c_\xe(x+r)^{a-\Qdel+\xe} = x^{a-\xe} \lh(C - c_\xe x^{2\xe-\Qdel} (1+r/x)^{a-\Qdel+\xe}\rh).
        \]
    Since $2\xe-\Qdel<0$ and $a-\Qdel+\xe < a-\xe$ by assumption, the assertion that $g(x) = \xO_+^*(x^{a-\xe})$ follows by taking $x$ to be sufficiently large, say $x > x_0(r)$. That $g(x) = \xO_-^*(x^{a-\xe})$ is proved similarly.
\end{proof}

The following proposition provides a template for our biasymptotic results.

\begin{proposition}
    \label{B:prop:Biasymp}
    Let $f:\nn\to\signs$. Suppose that there exists $\xa\in(0,1]$, $A \in (0,\fr12)$, and real functions $\phi(x)$ and $\phi_*(x)$ such that: As $x\to\infty$ one has
        \begin{subequations}
        \begin{align}
            \label{B:eq:pxfForm}
            p(x,f) &= e^{\phi(x)} \Big[ e^{\phi_*(x)} + (-1)^x e^{-\phi_*(x)} + O(x^{-\nfr15}e^{|\phi_*(x)|}) \Big], \\
            \label{B:eq:phiForm}
            \phi(x) &= \xa\xk\rt{x} + O(x^{A+\xe}), \\
            \label{B:eq:phi*BigO}
            \phi_*(x) &= O(x^{A+\xe}), \\
            \label{B:eq:phi*Omega*}
            \phi_*(x) &= \xO^*_\pm(x^{A-\xe}).
        \end{align}
        \end{subequations}
    Then $\spnf$ is biasymptotic with factor $\xa$.
\end{proposition}

\begin{proof}
    Let $L\in\nn$, let $0<\xe<A$, and let $C$ be a constant satisfying the $\xO^*_+$ condition for $\phi_*(x)$ in equation \eqref{B:eq:phi*Omega*}. In addition, fix an arbitrarily small $\xd>0$ to play the role of $\xe$ in the definition of the relation $\!\!\simstar$.
    
    Let $r > 0$ and suppose that $x$ is such that $\phi_*(y) > C x^{A-\xe}$ for $y \in [x,x+r)$. Further supposing that $x$ is large enough that $e^{-C x^{A-\xe}} < (x+r)^{-1/5}$, it follows that $e^{-\phi_*(y)} < y^{-1/5}$ for all $y \in [x,x+r)$. For all integers $n \in [x,x+r)$ then, one has
        \<
            \label{B:eq:pnfQuantity}
            p(n,f) = e^{\phi(n)+\phi_*(n)}(1+O(n^{-\nfr15}))
        \>
    and
        \<
            \label{B:eq:LogscAsymp}
            \logsc p(n,f) = \phi(n) + \phi_*(n) + O(n^{-\nfr15}) = \xa\xk\rt{n} + O(n^{A+\xe}).
        \>
    Equation \eqref{B:eq:LogscAsymp} shows that on this same interval, one has
        \<
            \label{B:eq:LogscAsymp2}
            \lh| \fr{\logsc p(n,f)}{\xa\xk\rt{n}} - 1 \rh| = O(n^{A-\nfr12+\xe}),
        \>
    and the assertion that $\logsc p(n,f) \simstar \xa\xk\rt{n}$ follows by taking $r$ large enough that $[x,x+r)$ contains a string of $L$ integers and taking $x$ large enough that the quantity \eqref{B:eq:LogscAsymp2} is less than $\xd$ for all $n \in [x,x+r)$.
    
    Similarly, if $\phi_*(y) < -C x^{A-\xe}$ for all $y \in [x,x+r)$ then one deduces that
        \[
            \begin{aligned}
                p(n,f) = (-1)^n e^{\phi(n)+\phi_*(n)} (1 + O(n^{-\nfr15})) \qquad (n \in [x,x+r))
            \end{aligned}
        \]
    in place of \eqref{B:eq:pnfQuantity}, and the assertion that $(-1)^n \logsc p(n,f) \simstar \xa\xk\rt{n}$ follows. 
\end{proof}

\subsection{\tops{Conditional uniform bounds for $|\xz(s)|^{-1}$.}}
\label{secsub:Littlewood}

We now discuss several lemmata towards establishing bounds on $|\xz(s)|^{-1}$ as $t \to \infty$ when $\xs$ is restricted to some compact interval $[\xs_0,\xs_1] \subset (\xQ,\infty)$, where we recall that
    \[
        \xQ = \sup\{ \Re(\xr) : \xz(\xr)=0 \}.  
    \]

A theorem of Littlewood \cite{littlewood1912} states that, on assumption of RH, i.e., on assumption that $\xQ = 1/2$, one has
    \<
        \label{eq:Littlewood}
        \log\xz(s) \less (\log t)^{2-2\xs+\xe}
    \>
uniformly for $1/2 < \xs_0 \leq \xs \leq 1$. That such an inequality holds with a modified exponent if one has $1/2 < \xQ < 1$ instead is apparent from the proof of \eqref{eq:Littlewood}. 

Despite this evidence and a number of similar conditional bounds appearing in elementary literature, we are unaware of a complete statement and proof of this generalization of Littlewood's result \eqref{eq:Littlewood} in the literature. Thus, in the appendix we include a proof of the following lemma following the proof of \eqref{eq:Littlewood} given in \cite{titchmarsh1986theory}*{Thm.\ 14.2}.

\begin{lemma}
	\label{lem:LittlewoodMod}
	Suppose that $\xQ < 1$. One has
		\<
			\label{eq:LittlewoodMod}
			|\log\xz(s)| \less (\log t)^{\fr{1-\xs}{1-\xQ}+\xe}
		\>
	uniformly for $\xQ < \xs_0 \leq \xs \leq 1$. 
\end{lemma}

For convenience we record the following two elementary bounds on $\xz(s)$ which are restatements of \cite{stein2003complex}*{Ch.\ 6 Prop.\ 2.7(i)} and \cite{stein2003complex}*{Ch.\ 7 Prop.\ 1.6}, respectively.

\begin{lemma}\label{lem:SSZeta}\leavevmode
    \begin{enumerate}[label=\normalfont{(\arabic*)}]
    \item \label{lem:SSZetaUpper} For all $\xs_0 \in [0,1]$ one has $\xz(s) \less t^{1-\xs_0+\xe}$ uniformly for $\xs\geq\xs_0$.
    \item \label{lem:SSZetaInv} For all $\xe > 0$ one has $1/\xz(s) \less t^\xe$ uniformly for $\xs\geq1$.
    \end{enumerate}
\end{lemma}

\begin{lemma}
    \label{lem:ZetaPowerBB}
	Suppose that $\xQ < 1$ and $\xs_0 > \xQ$. For all $\xe>0$ one has
        \begin{align*}
            \xz(s) \less t^\xe \qqand 1/\xz(s) \less t^\xe,
        \end{align*}
    with both inequalities uniform for $\xs\geq\xs_0$.
\end{lemma}

\begin{proof}
	For convenience let $\ell(\xs)=\fr{1-\xs}{1-\xQ}$. As $\ell(\xs_0)<1$, we fix $\xe > 0$ sufficiently small that $\ell(\xs_0)+\xe<1$, and by Lemma \ref{lem:LittlewoodMod} we have
		\[
			\log\xz(s) \less (\log{t})^{\ell(\xs)+\xe} \less (\log{t})^{\ell(\xs_0)+\xe}
		\] 
	uniformly for $\xs_0 \leq \xs \leq 1$. Since $\ell(\xs_0)+\xe < 1$, it follows that for any $\xe_1 > 0$ one has
		\<
            \label{eq:lzEps}
			\big|\.\log|\xz(s)|\big| \leq |\log\xz(s)| \leq \xe_1 \log{t} \qquad (\xs_0\leq\xs\leq1,\,\, t > t_0(\xs_0,\xe_1)).
		\>
	As both $\xz(s) \less t^\xe$ and $1/\xz(s) \less t^\xe$ uniformly for $\xs\geq1$ by Lemma \ref{lem:SSZeta}, the result follows at once from \eqref{eq:lzEps}.
\end{proof}

\begin{remark}
    Just as the truth of RH implies that of the Lindel\"of Hypothesis (see, e.g., \cite{titchmarsh1986theory}*{Ch.\ 14 and 15}), Lemma \ref{lem:ZetaPowerBB} shows that the ``partial Riemann Hypothesis'' that $\xQ<1$ implies a kind of ``partial Lindel\"of Hypothesis.'' 
\end{remark}

Finally we record the useful bound \cite{paris2001asymptotics}*{ineq.\ (2.1.19)}
	\<
		\label{eq:GammaBB}
		|\Gamma(s)| \less |s|^{\xs-\fr12} \expp{-\tf12 \pi |t| + \tf16|s|^{-1}} \qquad (\xs \geq 0).
    \>

\section{The reduction of \tops{$\phi(x,\xm)$ and $\phi_*(x,\xm)$}}
\label{sec:Mpp*}

We now focus on the functions $\phi(x,\xm)$ and $\phi_*(x,\xm)$ of equation \eqref{M:eq:pxmuForm}, often abbreviating them as $\phi(x)$ and $\phi_*(x)$, respectively. In this and the following section, let $x$ be sufficiently large and let $\xr$ and $\xr_*$ be the unique solutions to \sdpeqs, respectively. 

We recall from Lemma \ref{M:lem:pp*} the formulae
	\begin{align}
        \label{M:eq:phixRecall}
		\phi(x) &= \fr12 \lh[ \Psi(\xr)+\Psi(-\xr_*) + \fr{x(X+X_*)}{XX_*} - \fr12 \log\pth{4\pi^2 \PsiD{2}(\xr)\PsiD{2}(-\xr_*)} \rh], \\
        \label{M:eq:phi*xRecall}
		\phi_*(x) &= \fr12 \lh[ \Psi(\xr) - \Psi(-\xr_*) - \fr{x(X-X_*)}{XX_*} - \fr12 \log\pfrac{\PsiD{2}(\xr)}{\PsiD{2}(-\xr_*)} \rh].
    \end{align}
Our first step in analyzing these functions is to eliminate the implicit dependencies on $x$ above in favor more explicit dependencies. 
    
To examine these dependencies we briefly recall the arguments and consequences of \cite{daniels2023mobius}*{Prop.\ 7.1}, namely that for fixed $A>0$ and for sufficiently large $X>X_0(A)$ (regardless of the equations \eqref{eq:sdp} and \eqref{eq:sdp*} vis-\`a-vis $\xr = \exp(-1/X)$) one has
    \<
    \label{eq:PsiD1Int}
    \begin{aligned}
        \PsiD{1}(\xr) &= \ldint{2} 2^{-s-1}\fr{\xz(s)\xz(s+1)}{\xz(2s)}\xG(s+1)X^{s+1} \,ds \\
        & \qquad+ \ldint{2}(1-2^{-s-1})\fr{\xz(s+1)}{\xz(s)}\xG(s+1)X^{s+1} \,ds.
    \end{aligned}
    \>
Shifting the line of integration of the first integral left to $\xs = 1/2+\xe$, we pick up a residue of $\fr14X^2$ at $s=1$ and show the remaining integral on $\xs=1/2+\xe$ to be $O(X^{3/2+\xe})$ for a total contribution of $\fr14X^2 + O(X^{3/2+\xe})$. For the second integral in \eqref{eq:PsiD1Int}, we shift the contour of integration into the classical zero free region for $\xz(s)$, namely the region where
    \[
        |t|>t_0 \qqand \xs \geq 1 - \fr{c}{\log|t|}
    \]
for some universal constant $c>0$. Shifting the contour of integration into this region for $|t| \leq \exp(\rt{\logX})$ and integrating along $\xs=1$ otherwise, we show the second integral in \eqref{eq:PsiD1Int} to be $\less X^2(\logX)^{-A}$ and thereby conclude that
    \<
        \label{eq:PsiD1}
        \Psi_{(1)}(\xr) = \tf14X^2 + O(X^2(\logX)^{-A}) + O(X^{\fr32+\xe}).
    \>

Considering these arguments, we see that if it holds that $\xQ < 1$, then the second integral in \eqref{eq:PsiD1Int} is then $O(X^{1+\xQ+\xe})$, and therefore in place of \eqref{eq:PsiD1} we have the better estimate 
    \[
        \Psi_{(1)}(\xr) = \tf14X^2 + O(X^{1+\xQ+\xe}). 
    \]

Supposing now that $\PsiD{1}(\rho) = x$, we have $x=\tf14X^{2}[1+O(X^{\xQ-1+\xe})]$. Provided that $\xe$ is small enough that $\xQ-1+\xe<0$, it follows that for $X>X_0$ we have
    \[
        X = 2 \rt{x} \lh[ 1 + O(X^{\xQ-1+\xe}) \rh]^{-\fr12} = 2\rt{x}[1 + O(X^{\xQ-1+\xe})],
    \]
and moreover that
    \<
        \label{eq:Xx}
        X = 2\rt{x}\lh[1+O(x^{\fr12(\xQ-1)+\xe})\rh].
    \>

Similarly considering $\Psi_{(1)}(-\xr_*)$ we find that 
    \[
        \PsiD{1}(-\xr_*) = \tf14X_*^2[1+O(X_*^{\xQ-1+\xe})],
    \]
and assuming that $\PsiD{1}(-\xr_*)=x$ we similarly deduce that
    \<
        \label{eq:X*x}
        X_* = 2\rt{x} \lf[ 1 + O(x^{\fr12(\xQ-1)+\xe}) \rh].
    \>

As the quantity $\fr12(1-\xQ)$ appears frequently in our discussions, we let
    \<
        \xvq := \tf12(1-\xQ).
    \>
Considering \eqref{eq:Xx} and \eqref{eq:X*x}, we define $E = E(x)$, $E_* = E_*(x)$, and $\cE=\cE(x)$ via
    \<
        \label{eq:EE*}
        X = 2\rt{x}(1+E), \quad X_* = 2\rt{x}(1+E_*), \quad\text{and}\quad \cE = \max\{|E|,|E_*|\}
    \>
and summarize the above discussions in the following immediate lemma. 

\begin{lemma}
    \label{lem:EBounds}
    Suppose that $\xQ<1$, let $\xvq = \tf12(1-\xQ)$, and for sufficiently large $x$ let $\xr$ and $\xr_*$ be the solutions to \eqref{eq:sdp} and \eqref{eq:sdp*}, respectively. Defining $E(x)$, $E_*(x)$, and $\cE(x)$ via
        \[
            X = 2\rt{x}(1+E(x)), \quad X = 2\rt{x}(1+E_*(x)), \quad\text{and}\quad \cE(x) = \max\{|E|,|E_*|\},
        \] 
    one has 
        \<
            \label{M:eq:cEBB}
            \cE(x) \less x^{-\xvq+\xe}.
        \>
\end{lemma}

\begin{remark}
    \label{rem:EE*}
    As equation \eqref{eq:PsiD1} implies that $\fr14X^2 + O(X^2(\logX)^{-A})$, using arguments similar to those deriving \eqref{eq:Xx} and \eqref{eq:X*x} we deduce that, unconditionally, we have
        \[
            \cE(x) \less (\logx)^{-A} \qquad (A>0),
        \]
    whereby $\cE(x) = o(1)$ unconditionally.
\end{remark}

Going forward we always assume that $\xQ<1$ (so $\xvq > 0$), and we now examine the terms in formulae \eqref{M:eq:phixRecall} and \eqref{M:eq:phi*xRecall}. Recalling the relations of Theorem \ref{M:thm:Relations}, we have (unconditionally) that $\PsiD{2}(\rho) = 4x^{3/2}(1+o(1))$ and $\PsiD{2}(-\rho_*) = 4x^{3/2}(1+o(1))$, whereby
	\begin{align}
		\label{M:eq:LogPsiProd}
        \tf12 \log \lf( 4\pi^2\PsiD{2}(\rho)\PsiD{2}(-\rho_*) \rh) &= \tf32\logx + \log(8\pi) + o(1), \\
        \label{M:eq:LogPsiQuot}
		\tf12 \log \lf( \PsiD{2}(\rho)/\PsiD{2}(-\rho_*) \rh) &= \tf12 \log(1+o(1)) = o(1).
    \end{align}
Next, using the definitions \eqref{eq:EE*} we have
    \[
        \fr{x(X \pm X_*)}{XX_*} = \fr{\fr14(X \pm X_*)}{1+E+E_*+EE_*} = \tf14(X \pm X_*) \big[ 1+O(\cE) \big],
    \]
and using Lemma \ref{lem:EBounds} and the equality $1/2-\xvq = \xQ/2$ we deduce that
    \<
        \label{M:eq:XX*p}
        \fr{x(X+X_*)}{XX_*} = \tf14(X+X_*) + O(x^{1/2-\xvq+\xe}) = \rt{x} + O(x^{\xQ/2+\xe}).
    \>
Since $X-X_* = O(x^{1/2}\cE)$ we similarly deduce that
    \<
        \label{M:eq:XX*m}
        \fr{x(X-X_*)}{XX_*} = \tf14(X-X_*) + O(x^{1/2}\cE^2) = \tf14(X-X_*) + O(x^{\xQ/2-\xvq+\xe}),
    \>
and remark that the term $\tf14(X-X_*)$ here is intentionally left as-is for now.

We now examine the quantities $\Psi(\rho) \pm \Psi(-\rho_*)$. Breaking the sum $\sum_{n=1}^\infty (-1)^n \xm(n)n^{-s}$ into sums over even and odd $n$, one quickly deduces that
    \<
        \label{eq:Dirmun1}
        \Dir{(-1)^n \xm(n)}{s} = - \pfrac{1+2^{-s}}{1-2^{-s}}\fr{1}{\xz(s)},  
    \>
and it follows from \cite{daniels2023mobius}*{eqs.\ 6.4 and 6.5} (which are similar to \eqref{eq:PsiD1Int} above) and the absolute convergence of $\Psi(z)$ in \eqref{eq:Phi} that
    \<
    \label{M:eq:phiInt}
    \begin{aligned}
        &\Psi(\rho) \pm \Psi(-\rho_*) \\
        & \qquad = \tf14(X \pm X_*) + \vdint{\xs} (1-2^{-s-1})\fr{\xz(s+1)}{\xz(s)}\xG(s)\pth{X^s \mp \pfr{1+2^{-s}}{1-2^{-s}}X_*^s} \,ds \\
        & \qquad \qquad + \vdint{\xs'} 2^{-s-1} \xz(s+1)\fr{\xz(s)}{\xz(2s)}\xG(s) (X^s \pm X_*^s) \,ds \qquad (\xs>\xQ,\,\,\, \tf{\xQ}{2}<\xs'<1).
    \end{aligned}
    \>
In the interest of simplifying equation \eqref{M:eq:phiInt}, we write
    \[
        X^s \mp \pfr{1+2^{-s}}{1-2^{-s}} X_*^s = \fr{X^s \mp X_*^s - 2^{-s}(X^s \pm X_*^s)}{1-2^{-s}}  
    \]
and define
    \begin{align*}
        f(s) = \fr12 \pfrac{1-2^{-s-1}}{1-2^{-s}} \fr{\xz(s+1)}{\xz(s)}\xG(s) \qqand 
        g(s) = 2^{-s-2} \fr{\xz(s+1)\xz(s)}{\xz(2s)}\xG(s),
    \end{align*}
so that equation \eqref{M:eq:phiInt} becomes
    \<
        \label{M:eq:pp*Intfg}
        \begin{aligned}
            \Psi(\rho) \pm \Psi(-\rho_*) &= \tf14(X \pm X_*) + \vdint{\xs} 2 f(s) \pth{X^s \mp X_*^s - 2^{-s}(X^s \pm X_*^s)} \,ds \\
            & \qquad + \vdint{\xs'} 2 g(s) \pth{X^s \pm X_*^s} \,ds \qquad (\xs>\xQ,\,\,\, \tfr{\xQ}{2}<\xs'< 1).
        \end{aligned}    
    \>

Now separately considering $\Psi(\rho)+\Psi(-\rho_*)$ and $\Psi(\rho)-\Psi(-\rho_*)$, we combine equations \eqref{M:eq:LogPsiProd}, \eqref{M:eq:XX*p}, and \eqref{M:eq:pp*Intfg} and recall the factor of $1/2$ in \eqref{M:eq:phixRecall} to deduce that
    \<
        \label{M:eq:phixfg}
        \begin{aligned}
            \phi(x) &= \rt{x} + \vdint{\xs} f(s) \lf( X^s-X_*^s - 2^{-s}(X^s+X_*^s) \rh) \,ds + \vdint{\xs'} g(s) (X^s+X_*^s) \,ds  \\
            & \qquad\qquad\qquad + O(x^{\xQ/2+\xe}) \qquad (\xs>\xQ,\,\,\, \tfr{\xQ}{2}<\xs'<1).
        \end{aligned}
    \>
The asymptotic behavior of $\phi(x)$ is dominated by the term $\rt{x}$ in equation \eqref{M:eq:phixfg}: Indeed, taking $\xs = \xQ + 2\xe$ and $\xs' = \xQ/2 + 2\xe$ there and bounding
    \begin{align*}
        &\vdint{\xQ+2\xe} f(s)\pth{X^s-X_*^s - 2^{-s}(X^s+X_*^s)} \,ds \less X^{\xQ+2\xe} + X_*^{\xQ+2\xe} \less x^{\xQ/2 + \xe}, \\
        &\vdint{\xQ/2+2\xe} g(s)\pth{X^s+X_*^s} \,ds \less X^{\xQ/2+2\xe} + X_*^{\xQ/2+2\xe} \less x^{\xQ/4 + \xe},
    \end{align*}
we deduce that
    \<
        \label{M:eq:phixAsymp}
        \phi(x) = \rt{x} + O(x^{\xQ/2+\xe}).
    \>
We note that \eqref{M:eq:phixAsymp} implies $\phi(x)=\rt{x}+o(\rt{x})$, which is consistent with Lemma \ref{M:lem:pp*}.

\subsection{The reduction of \tops{$\Psi(\rho)-\Psi(-\rho_*)$}.}

We now turn to $\Psi(\xr)-\Psi(-\xr_*)$. Similar to the derivation of \eqref{M:eq:phixfg}, we combine equations \eqref{M:eq:LogPsiQuot}, \eqref{M:eq:XX*m}, and \eqref{M:eq:pp*Intfg}, noting that the term $\tf14(X-X_*)$ in \eqref{M:eq:pp*Intfg} is eliminated by the matching term in \eqref{M:eq:XX*m}, and find that
    \<
        \label{M:eq:phi*xfg}
        \begin{aligned}
            \phi_*(x) 
            &= \vdint{\xs} f(s)\pth{X^s+X_*^s - 2^{-s}(X^s-X_*^s)} \,ds + \vdint{\xs'} g(s)(X^s - X_*^s) \,ds \\
            & \qquad\qquad + O(x^{\xQ/2-\xvq+\xe}) \qquad (\xs>\xQ,\,\,\, \tf{\xQ}{2}<\xs'<1).
        \end{aligned}
    \>

Unlike for $\phi(x)$ in equation \eqref{M:eq:phixfg}, it is not evident in equation \eqref{M:eq:phi*xfg} what the asymptotic main term for $\phi_*(x)$ is, so we now treat the integrals in \eqref{M:eq:phi*xfg}, recalling that
    \[
        X = 2\rt{x}(1+E(x)), \quad X_* = 2\rt{x}(1+E_*(x)), \quad\text{and}\quad \cE = \max\{|E|,|E_*|\},
    \]
and that $\cE(x) \less x^{-\xvq+\xe}$.

By the bound on $\cE(x)$ we have $|E(x)|,|E_*(x)| < \fr12$ for $x>x_0$, whereby
    \<
        L(x) := \log(1+E(x)) \qqand L_*(x) := \log(1+E_*(x))
    \> 
are well-defined for sufficiently large $x$, and for these $x$ we have $|L(x)|,|L_*(x)|<1$. Thus, for sufficiently large $x$ we have the series expansions
	\begin{align}
		\label{eq:XexpSeries}
		X^s &= 2^s x^{\nfr{s}{2}} (1+E(x))^s = 2^s x^{\nfr{s}{2}} \sum_{k=0}^\infty \fr{s^k L^k(x)}{k!}, \\
        \label{eq:X*expSeries}
        X_*^s &= 2^s x^{\nfr{s}{2}} (1+E_*(x))^s = 2^s x^{\nfr{s}{2}} \sum_{k=0}^\infty \fr{s^k L_*^k(x)}{k!},
	\end{align}
with both series converging absolutely for fixed $s$.

We first consider the integral $\vdint{\xs} f(s)X^s\,ds$ from \eqref{M:eq:phi*xfg} (we recall $f(s)$ just below), using series \eqref{eq:XexpSeries} to write
    \<
        \label{M:eq:fXInt}
        \vdint{\xs} f(s)X^{s} \,ds = \vdint{\xs} 2^s f(s) \lf[ \sum_{k=0}^\infty \fr{s^kL^k(x)}{k!} \rh] x^{\nfr{s}{2}}\,ds.
    \>
Defining for $k \geq 0$ the functions
    \<
        \label{eq:Fk}
        F_k(s) = 2^s \big[ f(s) \big]  s^k = 2^s \lf[ \fr12\pfr{1-2^{-s-1}}{1-2^{-s}} \fr{\xz(s+1)}{\xz(s)} \xG(s) \rh] s^k,
    \>
we momentarily ignore convergence issues to formally write (for $\xs>\xQ$)
    \<
        \label{M:eq:fXIntFormal}
        \vdint{\xs} f(s)X^{s} \,ds = \vdint{\xs} F_0(s)x^{s/2} \,ds + \vdint{\xs} \sum_{k=1}^\infty \fr{F_k(s)L^k(x)}{k!}x^{s/2} \,ds. 
    \>

We use two lemmata to justify equation \eqref{M:eq:fXIntFormal}.

\begin{lemma}
    \label{M:lem:FInt}
    Suppose that $\xQ < 1$ and let $\xQ < \xs_0 \leq \xs_1 < \infty$. For all $\xs\in[\xs_0,\xs_1]$, as $k\to\infty$ one has
        \<
            \label{M:eq:FInt}
            \fr{1}{k!} \int_{-\infty}^\infty \lh| F_k(\xs+it) \rh| \,dt \less_{\xs_0,\xs_1} \exp(-\tf1{11}k).
        \>
\end{lemma}

\begin{proof}
    For simplicity we omit the subscripts $\xs_0,\xs_1$ from our inequalities below.
    
    As $\xQ<1$, by Lemma \ref{lem:ZetaPowerBB} one has $|\xz(\xs+it)|^{-1} \less t^\xd$ for any $\xd>0$, uniformly for $\xs_0 \leq \xs \leq \xs_1$. That is, for some positive $t_0 = t_0(\xs_0,\xs_1)$ and $c_0 = c_0(\xs_0,\xs_1)$ one has
        \<
            \label{M:eq:ZetaIneq}
            |\xz(\xs+it)|^{-1} \leq c_0 t^{\nfr12} \qquad (\xs_0 \leq \xs \leq \xs_1,\,\, t>t_0).
        \>
    As this inequality clearly remains valid with $\max\{t_0,1\}$ in place of $t_0$, we assume without loss of generality that $t_0 \geq 1$.

    We consider the integral in \eqref{M:eq:FInt} without the factor $1/k!$. Because $|F_k(\xs+it)|=|F_k(\xs-it)|$, we have
    	\<
            \label{eq:FInt}
    		\int_{-\infty}^\infty \ABS{F_k(\xs+it)} \,dt \less \int_0^\infty \afrac{(\xs+it)^k\xG(\xs+it)}{\xz(\xs+it)} \,dt.
        \> 
    With $t_0$ as in \eqref{M:eq:ZetaIneq}, let $T > \max\{\xs_1,t_0\}$, and suppose that $k > \pi T/2$. As $\xG(s)$ and $1/\xz(s)$ are bounded on $[\xs_0,\xs_1]\times[0,T]$ one has
    	\<
            \label{eq:FIntLower}
    		\int_{0}^{T} \afrac{(\xs+it)^k\xG(\xs+it)}{\xz(\xs+it)} \,dt 
    		\less \int_0^{T} (\xs^2 + t^2)^{k/2} \,dt \less 2^{k/2} T^{k+1},
    	\>
    and it remains to consider the right side of \eqref{eq:FInt} for $t > T$. Here, by \eqref{eq:GammaBB} one has
        \[
            |(\xs+it)^k\xG(\xs+it)| \less 2^{k/2} t^{k+\xs-1/2} e^{-\pi t/2}.
        \]
    Using \eqref{M:eq:ZetaIneq} then, we have
    	\<
    		\label{eq:FIntUpper}
    		\int_T^\infty \afrac{(\xs+it)^k\xG(\xs+it)}{\xz(\xs+it)} \,dt \less 2^{k/2} \int_T^\infty e^{-\pi t/2} t^{k+\xs} \,dt = 2^{k/2} \int_T^\infty e^{-k\vphi(t)}t^{\xs} \,dt,
    	\>
    where $\vphi(t) := \pi t/(2k) - \log{t}$. As $\vphi(t)$ has its global minimum at $t_* := 2k/\pi$, which is in the range $(T,\infty)$ by assumption, and $\vphi''(t_*) = \pi^2/(4k^2)$, an application of Laplace's method for integrals (see, e.g., \cite{stein2003complex}*{App.\ A}) shows that 
    	\begin{align}
            \label{eq:FIntSP} 
    		\int_T^\infty e^{-k\vphi(t)}t^{\xs} \,dt \less e^{-k \vphi(t_*)} \fr{t_*^{\xs}}{(k\vphi''(t_*))^{1/2}} 
            & \less k^{\xs} \exp(k\log{k}-k+k\log(\tfr{2}{\pi})).
    	\end{align}
    
    Combining inequalities \eqref{eq:FIntLower}--\eqref{eq:FIntSP} and dividing by $k!$, we have
        \[
            \fr{1}{k!} \int_{-\infty}^\infty |F_k(\xs+it)| \,dt 
            \less \fr{1}{k!}\.\Big[2^{k/2}T^{k+1} + 2^{k/2}k^{\xs}\expp{k\log{k}-k+k\log(\tfr{2}{\pi})}\Big].
        \]
    Using Stirling's well known formula $k! \asymp k^{1/2} e^{k\log{k}-k}$, it follows that
        \[
            2^{k/2}T^{k+1}/k! \less k^{-3/2} \exp\!\big(\!-\!k\log{k} + k + k\log(T\rt{2})\big)\less \exp(-\tf12 k\log{k})
        \]
    and
        \[
            (2^{k/2}k^\xs/k!) \exp\!\big(k \log{k} - k + k\log(\tfr{2}{\pi})\big) \less k^{\xs} \exp((\tf32\log{2}-\log\pi)k).
        \]
    Computing that $\tf32\log2 - \log\pi \approx -0.105 < -\tf{1}{10}$, we conclude that
    	\<
            \label{M:eq:FIntFinal}
            \fr{1}{k!} \int_{-\infty}^\infty |F_k(\xs+it)| \,dt
            \less k^{\xs}e^{-\fr{1}{10}k} + e^{-\fr12 k\log{k}} \less e^{-\fr1{11}k}.
    	\>
    
    With similar arguments we easily see that $\int_{-\infty}^\infty |F_k(\xs+it)| \,dt < \infty$ for $k \leq \pi T/2$, and the result follows at once by enlarging the constant in \eqref{M:eq:FIntFinal} to account for these finitely many $k$.
\end{proof}

Recalling that $\xvq = \fr12(1-\xQ)$, we now establish the main term for $\vdint{\xs} f(x)X^s \,ds$.

\begin{lemma}
    \label{M:lem:fXIntMain}
    Suppose that $\xQ < 1$. For all $\xs > \xQ$ one has
        \<
            \vdint{\xs} f(s)X^{s} \,ds = \vdint{\xs} F_0(s)x^{s/2} \,ds + O(x^{\xQ/2-\Qdel+\xe}).
        \>
\end{lemma}

\begin{proof}
    We recall that (formally) we write
    \<
        \label{M:eq:fXIntSplit}
        \vdint{\xs} f(s)X^{s} \,ds = \vdint{\xs} F_0(s)x^{s/2} \,ds + \vdint{\xs} \sum_{k=1}^\infty \fr{F_k(s)L^k(x)}{k!} x^{s/2} \,ds 
    \>
    and that, by Lemma \ref{lem:EBounds}, we have $L(x) \less \cE(x) \less x^{-\xvq+\xe}$. Trivially bounding
    \[
        \vdint{\xs} \sum_{k=1}^\infty \fr{F_k(s)L^k(x)}{k!} x^{s/2} \,ds 
        \less x^{\xs/2} \sum_{k=1}^\infty \fr{|L(x)|^k}{k!} \int_{-\infty}^\infty |F_k(\xs+it)| \,dt,
    \]
we apply Lemma \ref{M:lem:FInt} to deduce that
    \begin{align}
        \label{M:eq:FIntSumBB}
        & \vdint{\xs} \sum_{k=1}^\infty \fr{F_k(s)L^k(x)}{k!} x^{s/2} \,ds 
            \less_{\xs} x^{\xs/2} \sum_{k=1}^\infty  x^{-k(\Qdel-\xe)} e^{-k/11} \less_{\xs} x^{\xs/2-\Qdel+\xe}. 
    \end{align}

Using similar arguments to those of Lemma \ref{M:lem:FInt} one easily sees that 
    \[
        \vdint{\xs} F_0(s)x^{s/2} \,ds \less_\xs x^{\xs/2} \qquad (\xs > \xQ),
    \] 
whereby, in light of inequality \eqref{M:eq:FIntSumBB}, equation \eqref{M:eq:fXIntSplit} is justified. In fact, as we have shown both integrals on the right side of equation \eqref{M:eq:fXIntSplit} to be absolutely convergent for all $\xs > \xQ$, and their integrands are analytic for $\xs > \xQ$, we need not use the same $\xs$ for both integrals. Selecting $\xs = \xQ + 2\xe$ for the latter integral in equation \eqref{M:eq:fXIntSplit} then, we apply inequality \eqref{M:eq:FIntSumBB} to deduce that
    \[
        \vdint{\xs} f(s)X^{s} \,ds = \vdint{\xs} F_0(s)x^{s/2} \,ds + O(x^{\xQ/2-\Qdel+2\xe}),    
    \]
and the result follows by using $\xe/2$ in place of $\xe$.
\end{proof}

The arguments of Lemma \ref{M:lem:fXIntMain} may be similarly applied to the integrals $\vdint{\xs} f(s)X_*^s\,ds$ and $\vdint{\xs} 2^{-s}f(s)(X^s-X_*^s) \,ds$ in \eqref{M:eq:phi*xfg}, and doing so we find that
    \<
        \label{M:eq:fXIntMain*}
        \vdint{\xs} f(s) \lf( X^s+X^s-2^{-s}(X^s-X_*^s) \rh) \,ds = \vdint{\xs} 2F_0(s) x^{s/2} \,ds + O(x^{\xQ/2-\Qdel+\xe}).
    \>

We now turn to $\vdint{\xs'} g(s)(X^s-X_*^s) \,ds$ in equation \eqref{M:eq:phi*xfg} (recalling $g(s)$ just below). Although a formula like that of Lemma \ref{M:lem:fXIntMain} holds for $\vdint{\xs'}g(s)X^s \,ds$, such a formula is unnecessary since the integral $\vdint{\xs'} g(s)(X^s-X_*^s) \,ds$ is ultimately part of the error term in \eqref{M:eq:phi*xfg} as we show now. 

Defining
    \[
        G_k(s) = 2^s \big[ g(s) \big] s^k = 2^{s} \lf[ 2^{-s-2}\fr{\xz(s+1)\xz(s)}{\xz(2s)}\xG(s) \rh]s^k     
    \] 
and recalling the series \eqref{eq:XexpSeries} and \eqref{eq:X*expSeries} we find that
    \<
        \label{M:eq:GIntFull}
        \vdint{\xs'} g(s)(X^s-X_*^s) \,ds = \vdint{\xs'} \lf[ \sum_{k=1}^\infty \fr{G_k(s)}{k!} (L^k(x)-L_*^k(x)) \rh] x^{s/2} \,ds \quad (\xs' > \tf12\xQ),
    \>
noting the cancellation of the $k=0$ terms.

Let $\fr{\xQ}{2} < \xs_0 \leq \xs_1 < 1$. In a manner similar to that of Lemma \ref{M:lem:FInt}, for all $\xs \in [\xs_0,\xs_1]$, as $k\to\infty$ one has
    \<
        \label{M:eq:GInt}
        \fr{1}{k!} \int_{-\infty}^\infty |G_k(\xs+it)| \,dt \less_{\xs_0,\xs_1} \exp(-\tf{1}{11}k).
    \>
Recalling that $L(x),L_*(x) \less x^{-\Qdel+\xe}$, equations \eqref{M:eq:GIntFull} and \eqref{M:eq:GInt} show that
    \[
        \vdint{\xs'} g(s)(X^s-X_*^s) \,ds \less_{\xs'} x^{\xs'/2} \sum_{k=1}^\infty x^{-k(\Qdel-\xe)} e^{-k/11} \less_{\xs'} x^{\xs'/2-\Qdel+\xe} \qquad (\tfr{\xQ}{2}<\xs'<1),
    \]
and selecting $\xs' = \xQ/2 + 2\xe$ we find that
    \<
        \label{M:eq:gXIntBB}
        \vdint{\xQ/2+2\xe} g(s)(X^s-X_*^s) \,ds \less x^{\xQ/4-\Qdel+2\xe},
    \>
which is smaller than the error term in \eqref{M:eq:fXIntMain*}.

Now returning to equation \eqref{M:eq:phi*xfg}, we use equations \eqref{M:eq:fXIntMain*} and \eqref{M:eq:gXIntBB} to deduce that
    \<
        \label{M:eq:phi*xIntF}
        \phi_*(x) = \vdint{\xs}2F_0(s)x^{s/2}\,ds + O(x^{\xQ/2-\Qdel+\xe}) \qquad (\xs > \xQ).    
    \>
We simplify our notation by defining (see \eqref{eq:Fk})
    \<
        \label{eq:F}
        F(s) = 2F_0(s) = 2 \lf[ 2^{s-1} \pfrac{1-2^{-s-1}}{1-2^{-s}} \fr{\xz(s+1)}{\xz(s)} \xG(s) \rh]        
    \>
and
    \<
        \label{M:eq:phi0}
        \phi_0(x,\xm) = \vdint{\xs} F(s) x^{s/2}\,ds \qquad (\xs>\xQ),
    \>
and summarize the results of this section in the following proposition.

\begin{proposition}
    \label{M:prop:pp*Asymp}
    Suppose that $\xQ<1$ and let $\Qdel=\tf12(1-\xQ)$. As $x\to\infty$, one has
        \[
            \phi(x,\xm) = \rt{x} + O(x^{\xQ/2+\xe})    
        \]
    and
        \[
            \phi_*(x,\xm) = \phi_0(x,\xm) + O(x^{\xQ/2-\Qdel+\xe}),
        \]
    where
        \[
            \phi_0(x,\xm) = \vdint{\xs} F(s)x^{s/2} \,ds = \vdint{\xs} 2^s \pfrac{1-2^{-s-1}}{1-2^{-s}} \fr{\xz(s+1)}{\xz(s)} \xG(s)x^{s/2} \,ds \qquad (\xs>\xQ). 
        \]
    In addition, one has $\phi_*(x,\xm) \less x^{\xQ/2+\xe}$. 
\end{proposition}

\begin{proof}
    Only the final statement of the proposition remains to be shown at this point, and said statement follows immediately from taking $\xs = \xQ+2\xe$ in equation \eqref{M:eq:phi*xIntF}.
\end{proof}

\section{\tops{Biasymptotic results for $(p(n,\xm))_\nn$}}
\label{sec:MBiasymp}

At this point we have seen that, on assumption that $\xQ < 1$, the first three equations listed in Proposition \ref{B:prop:Biasymp} are valid for $p(x,\xm)$, $\phi(x,\xm)$, and $\phi_*(x,\xm)$ using $\xa = \xk^{-1} = \rt{6}/(2\pi)$ and $A = \xQ/2$. It thus remains to verify that the last equation therein is also valid, i.e., that $\phi_*(x,\xm) = \xO_{\pm}^*(x^{\xQ/2-\xe})$. Considering Lemma \ref{B:lem:Omega*Approx} and Proposition \ref{M:prop:pp*Asymp}, we may establish said $\xO_{\pm}^*$--result by establishing a similar result for $\phi_0(x,\xm)$. 

In this section we often abbreviate $\phi_*(x,\xm)$ and $\phi_0(x,\xm)$ to $\phi_*(x)$ and $\phi_0(x)$, respectively.

\begin{remark}
    Although the majority of our analyses thus far deal with only large positive $x$, it is evident from \eqref{M:eq:phi0} that $\phi_0(x)$ may be defined for all $x > 0$, so we assume it to be thus defined using \eqref{M:eq:phi0}.
\end{remark}

Our first step is to show that $\phi_0(x) = \xO_{\pm}(x^{\xQ/2-\xe})$. For this we recall the following analogue of a theorem of Landau. 

\begin{lemma}[{\cite{montgomery2007multiplicative}*{Thm.\ 15.1}}]
	\label{lem:Landau}
	Suppose that $A(x)$ is a bounded Riemann-integrable function in any finite interval $1 \leq x \leq X$, and that $A(x) \geq 0$ for all $x > X_0$. Let $\xs_c$ denote the infimum of those $\xs$ for which $\int_{X_0}^\infty A(x) x^{-\xs}\,dx < \infty$. The function 
		\[
			\mathscr{A}(s) = \int_1^\infty A(x) x^{-s}\, dx
		\]
	is analytic in the half-plane $\xs > \xs_c$, but not at the point $s = \xs_c$. 
\end{lemma}

The function $F(s)$ from \eqref{eq:F} is clearly analytic for $\sigma > \xQ$, and using arguments like in the proof of Lemma \ref{M:lem:FInt} we readily deduce that $|F(s)| \to 0$ as $|t| \to \infty$, uniformly for $\sigma$ in compact intervals $[\xs_0,\xs_1] \subset (\xQ,\infty)$; the analogous properties hold for $2F(-2s)$ on the left halfplane $\sigma < -\xQ/2$. 

Using $\cM[f(x);s]$ and $\cM^{-1}[\vphi(s);x]$ to denote the Mellin and inverse Mellin transforms of functions $f(x)$ and $\vphi(s)$, respectively, we have
	\[
		\phi_0(x) = \MelInv{F(s)}{x^{-\fr12}} = \MelInv{2 F(-2s)}{x}.
	\]
By Mellin inversion then, it follows that
	\<
		\label{M:eq:phi0Mellin}
		\Mel{\phi_0(x)}{s} = 2 F(-2s) \qquad (\sigma < -\tfr{\xQ}{2}).
	\>

From equation \eqref{M:eq:phi0} it is evident that $\phi_0(x)\less_{\xs} x^{\xs/2}$ for all $\xs>\xQ$ and all $x>0$, and from this we deduce that
    \<
        \MelZero{\phi_0(x)}{s} := \int_0^1 \phi_0(x)x^{s-1} \,dx
    \>
converges for all $\xs>0$ and is analytic on that halfplane. We now follow arguments from \cite{montgomery2007multiplicative}*{Lem.\ 15.3} to establish an $\xO$--result for $\phi_0(x)$.

\begin{lemma}
    \label{M:lem:phi0Omega}
    If one has $\xQ<1$, then $\phi_0(x,\xm) = \xO_{\pm}(x^{\xQ/2 - \xe})$.
\end{lemma}

\begin{proof}
	Fix $\xe > 0$ and suppose by way of contradiction that 
		\<
			\label{eq:phi0OmegaContra}
			\phi_0(x) < x^{\xQ/2-\xe} \qquad \text{for $x > x_0$.}
		\>
	For $\xs > \xQ$ we have
		\[
			\int_1^\infty \phi_0(x)x^{-s/2-1} \,dx = \Mel{\phi_0(x)}{-s/2} - \MelZero{\phi_0(x)}{-s/2} \qquad (\xs > \xQ),
		\]
	which by equation \eqref{M:eq:phi0Mellin} implies (again for $\xs>\xQ$) that
		\begin{align}
			\label{M:eq:Omega}
			\int_1^\infty (x^{\xQ/2-\xe}-\phi_0(x)) x^{-s/2-1}\,dx 
			&= \fr{2}{s-\xQ+2\xe} - 2F(s) + \MelZero{\phi_0(x)}{-s/2}.
		\end{align}

	In addition to being analytic on the halfplane $\xs > \xQ$, the right side of equation \eqref{M:eq:Omega} is analytic at points $s = \xs + i0$ with $\xs \in (\xQ-2\xe, \infty)$ and is therefore analytic in some $\cc$-neighborhood of this interval. 

	However, if the the left side of \eqref{M:eq:Omega} is analytic on some $\cc$-neighborhood of $(\xQ-2\xe, \infty)$, then Lemma \ref{lem:Landau} implies that this integral is analytic on the entire halfplane $\xs > \xQ-2\xe$. This causes a contradiction, as the factor $1/\xz(s)$ of $F(s)$ has poles with real parts arbitrarily close to $\xQ$. The assertion \eqref{eq:phi0OmegaContra} must therefore be false, and the relation $\phi_0(x)=\xO_+(x^{\xQ/2-\xe})$ follows. That $\phi_0(x) = \xO_{-}(x^{\xQ/2-\xe})$ is proved similarly.
\end{proof}

Because $|F(s)| \to 0$ as $|t| \to \infty$, uniformly for $\xs$ in intervals $[\xs_0,\xs_1] \subset (\xQ,\infty)$, it is evident that $\phi_0(x)$ is differentiable with derivative   
    \<
        \label{M:eq:phi0Deriv}
        \fr12\,\vdint{\xs} sF(s)x^{s/2-1} \,ds \qquad (x>0,\,\,\xs>\xQ). 
    \>
Taking $\xs = \xQ + 2\xe$ we deduce that
    \<
        \label{M:eq:phi0DerivBB}
        |\phi_0'(x)| \leq c_\xe x^{\xQ/2-1+\xe} \qquad \text{for some $c_\xe>0$.}
    \>

We now establish Theorem \ref{thm:muBiasymp} (the formal version of statement \ref{it:mu}), namely: \emph{If $\xQ < 1$, then $\spnmu$ is biasymptotic with factor $\xk^{-1} = \pi^{-1}\rt{3/2}$.}

\begin{proof}[Proof of Thm.\ \ref{thm:muBiasymp}]
    It follows at once from Lemmata \ref{B:lem:DerivDecay} and \ref{M:lem:phi0Omega} and inequality \eqref{M:eq:phi0DerivBB} that $\phi_0(x) = \xO_{\pm}^*(x^{\xQ/2-\xe})$, whereby Lemma \ref{B:lem:Omega*Approx} and Proposition \ref{M:prop:pp*Asymp} imply that $\phi_*(x) = \xO_\pm^*(x^{\xQ/2-\xe})$. 
    Together with Lemma \ref{M:lem:pp*} and Proposition \ref{M:prop:pp*Asymp}, these facts show that the assumptions of Proposition \ref{B:prop:Biasymp} hold with $\xa = \xk^{-1}$ and $A = \fr12\xQ$, whereby $\spnmu$ is biasymptotic with factor $\xk^{-1}$.
\end{proof}

\section{The reduction of \tops{$\phi(x,\xl)$ and $\phi_*(x,\xl)$}}
\label{sec:Lpp*}

We now turn our investigations toward $p(x,\xl)$, $\phi(x,\xl)$, and $\phi_*(x,\xl)$. In much the same way that the analyses of $p(x,\xl)$ in section 12 of \cite{daniels2023mobius} compress the analyses of $p(x,\xm)$ in sections 4--11 there, this and the following sections reuse many of the results from sections \ref{sec:Mpp*} and \ref{sec:MBiasymp}. Going forward let $\Psi(z)$, $\phi(x)$, etc.\ refer to $\Psi(z,\xl)$, $\phi(x,\xl)$, etc.\ unless stated otherwise. In addition, we recall that 
    \[
        \fz = \xz(2) = \pi^2/6, \quad \xk = 2\rt{\xz(2)} = \pi\rt{2/3}, \quad\text{and}\quad \xvq = \tf12(1-\xQ).
    \]

We now state the analogue of Lemma \ref{M:lem:pp*} for $p(x,\xl)$. 
\begin{lemma}
    \label{L:lem:pp*}
    Fix $A>0$, let $\Psi(z)=\Psi(z,\xl)$, and for all sufficiently large $x$ let $\xr$ and $\xr_*$ be the unique solutions to equations \eqref{eq:sdp} and \eqref{eq:sdp*}, respectively. As $x\to\infty$ one has
        \<
            \label{L:eq:pnxlpp*}
            p(x,\xl) = e^{\phi(x)} \lf[ e^{\phi_*(x)} + (-1)^{-x} e^{-\phi_*(x)} + \O{x^{-1/5}e^{|\phi_*(x)|}} \rh],
        \>
    where $(-1)^{-x} = \exp(-\pi ix)$ and
        \begin{align}
            \label{L:eq:phi}
            \phi(x) &= \fr12 \lf[ \Psi(\xr)+\Psi(-\xr_*) + \frac{x(X + X_*)}{XX_*} - \fr12 \log\lf( 4\pi^2 \PsiD{2}(\xr)\PsiD{2}(-\xr_*) \rh) \rh], \\
            \label{L:eq:phi*}
            \phi_*(x) &= \fr12 \lf[ \Psi(\xr) - \Psi(-\xr_*) - \frac{x(X - X_*)}{XX_*} - \fr12 \log\pfrac{\PsiD{2}(\xr)}{\PsiD{2}(-\xr_*)} \rh].
        \end{align}
    Moreover, as $x\to\infty$ one has
        \begin{align*}
            \phi(x) &= \tf12\xk\rt{x} + O_A(\sqrt{x} (\logx)^{-A}), \\ 
            \phi_*(x) &= O_A(\sqrt{x} (\logx)^{-A}).
        \end{align*}
\end{lemma}

Again defining $E=E(x)$, $E_*=E_*(x)$, and $\cE=\cE(x)$ so as to satisfy
    \<
        \label{L:eq:EE*}
        X = (\xk/\fz)\rt{x}(1+E), \quad X_* = (\xk/\fz)\rt{x}(1+E_*), \quad\text{and}\quad \cE = \max\{|E|,|E_*|\},
    \>
the assumption that $\xQ < 1$ again implies the bound 
    \<
        \label{L:eq:cEBB}
        \cE(x) \less x^{-\xvq+\xe}.
    \>
Considering the overwhelming similarities in the statements of Theorems \ref{M:thm:asymp} and \ref{L:thm:pnlAsymp}, Lemmata \ref{M:lem:pp*} and \ref{L:lem:pp*}, and the relations \eqref{eq:EE*}, \eqref{M:eq:cEBB}, \eqref{L:eq:EE*}, and \eqref{L:eq:cEBB}, there would be little value in repeating the technical, notation-heavy derivations of section \ref{sec:Mpp*} for $\phi(x,\xl)$ and $\phi_*(x,\xl)$. Indeed, arguing as done for \eqref{eq:Dirmun1} and \eqref{M:eq:phiInt}, we find that
    \[
        \Dir{(-1)^n \xl(n)}{s} = -(1+2^{1-s})\fr{\xz(2s)}{\xz(s)} 
    \]
and that
    \[
    \begin{aligned}
        &\Psi(\rho) \pm \Psi(-\rho_*) \\
        & \qquad = \fr{\fz}{4}(X \pm X_*) + \vdint{\xs} (1-2^{-s-1})\xz(s+1)\fr{\xz(2s)}{\xz(s)}\xG(s) \Big[ X^s \mp (1+2^{1-s})X_*^s \Big] \,ds \\
        & \qquad \qquad + \vdint{\xs'} 2^{-s-1} \xz(s+1)\xz(s) \big[ X^s \pm X_*^s \big] \,ds \qquad (\xs > \xQ,\,\,\, 0 < \xs' < 1).
    \end{aligned}
    \]

By arguments nearly identical to those of section \ref{sec:Mpp*} we have
    \[
        \phi(x) = \tf12\xk\rt{x} + O(x^{\xQ/2+\xe}),   
    \]
and we then isolate the main term of $\phi_*(x)$ to find that
    \[
        \phi_*(x) = \vdint{\xs} (\xk/\fz)^s(1-2^{-s-1})\xz(s+1)\fr{\xz(2s)}{\xz(2)}\xG(s) x^{s/2} \,ds + O(x^{\xQ/2-\Qdel+\xe}) \qquad (\xs > \xQ).    
    \]
Defining 
    \[
        \xh(s) = (\xk/\fz)^s (1-2^{-s-1})\xz(s+1), \qquad H(s) = \xh(s)\fr{\xz(2s)}{\xz(s)}\xG(s),
    \] 
and 
    \[
        \phi_0(x,\xl) = \vdint{\xs} H(s)x^{s/2} \,ds \qquad (\xs>\xQ),
    \]
we have the following analogue of Proposition \ref{M:prop:pp*Asymp}.

\begin{proposition}
    \label{L:prop:pp*Asymp}
    Suppose that $\xQ<1$ and let 
        \[
            \fz=\pi^2/6, \quad \xk = \pi\rt{2/3}, \quad\text{and}\quad \Qdel=\tf12(1-\xQ).
        \]
    As $x\to\infty$ one has
        \<
            \phi(x,\xl) = \tf12 \xk\rt{x} + O(x^{\xQ/2+\xe})
        \>
    and
        \<
            \label{L:eq:phi*AsympRecall}
            \phi_*(x,\xl) = \phi_0(x,\xl) + O(x^{\xQ/2-\Qdel+\xe}),
        \>
    where
        \<
            \label{L:eq:phi0Recall}
            \phi_0(x,\xl) = \vdint{\xs} H(s)x^{s/2} \,ds = \vdint{\xs} \eta(s)\fr{\xz(2s)}{\xz(s)} \Gamma(s)x^{s/2} \,ds \qquad (\xs>\xQ) 
        \>
    and $\eta(s) = (\xk/\fz)^s(1-2^{-s-1})\xz(s+1)$. In addition, one has $\phi_*(x,\xl) \less x^{\xQ/2+\xe}$.
\end{proposition}

\section{\tops{Biasymptotic results for $(p(n,\xl))_\nn$}}
\label{sec:LBiasymp}

We now establish Theorem \ref{thm:nonRHBiasymp} (the formal version of statement \ref{it:xl}), namely: \emph{If $\xQ < 1$ but $\xQ>\tf12$, i.e., RH does not hold, then $\spnxl$ is biasymptotic with factor $\tf12$.}

As usual, here we often abbreviate $\phi_*(x,\xl)$ and $\phi_0(x,\xl)$ to $\phi_*(x)$ and $\phi_0(x)$, respectively. Repeating the Mellin-based arguments used for equation \eqref{M:eq:phi0Mellin} we find that
    \<
        \label{L:eq:phi0Mellin}
        \Mel{\phi_0(x)}{s} = 2 H(-2s) \qquad (\sigma < -\tfr{\xQ}{2}),
    \>
that $\MelZero{\phi_0(x)}{s} = \int_0^1 \phi_0(x)x^{-s-1} \,dx$ is analytic for $\xs>0$, and that
    \<
        \label{L:eq:Omega}
        \int_1^\infty \pth{ x^{\xQ/2-\xe} - \phi_0(x) } x^{-s/2-1} \,dx = \fr{2}{s-\xQ+2\xe} - 2H(s) + \mathcal{M}^0\lh[\phi_0(x);-s/2\rh]
    \>
for $\xs > \xQ$.

Despite the similarities of the above statements with those in section \ref{sec:MBiasymp}, the proof of Lemma \ref{M:lem:phi0Omega} cannot be used to deduce $\xO$--results for $\phi_0(x)$ if one has $\xQ=1/2$ since the factor $\xz(2s)$ of $H(s)$ has a simple pole at $s=1/2$. In particular, the right side of equation \eqref{L:eq:Omega} is not analytic on any (real) interval $(\xQ-2\xe,\infty)$ if $\xQ=1/2$. 

While partial results may be established when $\xQ=1/2$ using the methods of Lemma \ref{M:lem:phi0Omega}, for now we avoid the issue by making the stronger assumption that $1/2 < \xQ < 1$, allowing the proof of Lemma \ref{M:lem:phi0Omega} to be reused completely. The case of $\xQ=1/2$ is specially treated in section \ref{sec:Riemann}.

\begin{lemma}
    \label{L:lem:phi0Omega}
    If one has $\fr12 < \xQ < 1$, then $\phi_0(x,\xl) = \xO_{\pm}(x^{\xQ/2-\xe})$.
\end{lemma}

\begin{proof} 
    The proof is identical to that of Lemma \ref{M:lem:phi0Omega} with $F(s)$ replaced by $H(s)$.
\end{proof}

Maintaining the assumption that $\xQ < 1$ we find, in exactly the same manner as \eqref{M:eq:phi0Deriv} and \eqref{M:eq:phi0DerivBB}, that $\phi_0(x)$ is differentiable with  
    \<
        \label{L:eq:phi0Deriv}
        \phi_0'(x) = \fr12\,\vdint{\xs} sH(s)x^{s/2-1} \,ds \qquad (x>0,\,\, \xs > \xQ), 
    \>
and that for all $\xe > 0$ it holds that
    \<
        \label{L:eq:phi0DerivBB}
        |\phi_0'(x)| \leq c_\xe x^{\xQ/2-1+\xe} \qquad \text{for some $c_\xe > 0$ and all $x>0$.}
    \>

\begin{remark}
    Equation \eqref{L:eq:phi0Deriv} and inequality \eqref{L:eq:phi0DerivBB} do not require that $\xQ > \fr12$, and are thus used in later arguments when $\xQ=\fr12$ is assumed. 
\end{remark}

\begin{proof}[Proof of Thm.\ \ref{thm:nonRHBiasymp}]
    By Lemmata \ref{B:lem:DerivDecay} and \ref{L:lem:phi0Omega} and inequality \eqref{L:eq:phi0DerivBB} we have $\phi_0(x,\xl) = \xO_{\pm}^*(x^{\xQ/2-\xe})$, whereby Lemma \ref{B:lem:Omega*Approx} and Proposition \ref{L:prop:pp*Asymp} imply that $\phi_*(x,\xl) = \xO_\pm^*(x^{\xQ/2-\xe})$. Together with Lemma \ref{L:lem:pp*} and Proposition \ref{L:prop:pp*Asymp}, these facts show that the assumptions of Proposition \ref{B:prop:Biasymp} hold with $\xa = \fr12$ and $A = \tf12\xQ$, and the theorem follows.
\end{proof}

\section{The Riemann Hypothesis and \tops{$\spnxl$}}
\label{sec:Riemann}

As mentioned in statement \ref{it:asm} and section \ref{sec:Biasymp}, the Riemann Hypothesis (RH), i.e., the assertion that $\xQ = 1/2$, presents a delicate scenario vis-\`a-vis the (bi-)asymptotic type of $\spnxl$. In contrast to our results on $\spnmu$, on RH it is possible that
    \<
        \label{L:eq:asm}
        p(n,\xl) = (-1)^n \expp{ \tf12 \xk\rt{n}+o(\rt{n}) } \qquad (\xk = \pi\rt{2/3},\,\, n \to \infty),
    \>
that is, the sequence $((-1)^n p(n,\xl))_\nn$ is asymptotic with factor $1/2$. This behavior is not guaranteed by RH alone though; rather, several conditions must be satisfied. 

One such condition is the Simplicity Hypothesis (SH), i.e., the hypothesis that all zeros of $\xz(s)$ are simple. Should SH hold, the residues of $1/\xz(s)$ at the zeros of $\xz(s)$ on the \emph{critical line} $\xs=1/2$ must also be considered. In particular, one requires an assumption of effective simplicity (ES), namely that for some $C > 0$ one has
    \<
        \label{eq:ZetaDerivBB}
        \lh| 1 / \xz'(\rho)\rh| \leq C |\rho| \tag{ES}
    \> 
for all zeros $\rho = \nfr12 + i\xg$ of $\xz(s)$. In addition we discuss a condition of linear independence of the imaginary parts $\xg$ of these zeros in section \ref{sec:RiemannNS}.

Considering the result of Theorem \ref{thm:nonRHBiasymp}, going forward we assume RH unless stated otherwise, and we write the zeros $\rho$ of $\xz(s)$ on the critical line as $\rho = \nfr12 + i\xg$; the trivial zeros $-2,-4,-6,\ldots$ of $\xz(s)$ are not referenced. The quantities $\xg$ are the \emph{ordinates} of the zeros $\rho$, and $\rho$ and $\xg$ always denote such zeros and ordinates, respectively. In particular, for sums $\sum_\xr$ one lets $\rho$ run over the zeros $\rho = \nfr12 + i\xg$, and similarly for sums $\sum_\xg$ over the ordinates $\xg$. We note that although $\rho$ is used in previous sections to denote a solution of \sdpeqs, this does not cause confusion.

We now present informal arguments to motivate our later discussions, assuming SH for the remainder of this section. As $\tf{\xQ}{2} = \tf14$ and $\Qdel = \tf12(1-\Theta) = \tf14$ on RH, Proposition \ref{L:prop:pp*Asymp} implies that
    \<
        \label{R:eq:phi*Reduc}
        \phi_*(x) = \phi_0(x) + O(x^{\xe})
    \> 
with
    \<
        \label{R:eq:phi0Reduc2}
        \phi_0(x) = \vdint{\xs} H(s)x^{s/2} \,ds = \vdint{\xs} \eta(s) \fr{\xz(2s)}{\xz(s)} \Gamma(s) x^{s/2} \,ds \qquad (\xs > \tf12),
    \> 
where we recall that
    \[
        \fz = \pi^2/6, \quad \xk=\pi\rt{2/3}, \quad\text{and}\quad \eta(s) = (\xk/\fz)^s(1-2^{-s-1})\xz(s+1).
    \]
    
For all zeros $\rho=\nfr12+i\xg$ of $\xz(s)$ we define $\fc_\rho$ via
    \[
        \fc_\rho = \Res_{s=\rho}H(s) = \eta(\rho) \fr{\xz(2\rho)}{\xz'(\rho)} \Gamma(\rho),
    \]
where the last equality is valid by SH. Observing that the simple pole of $H(s)$ at $s=\tf12$ is due to $\xz(2s)$ rather than $1/\xz(s)$, we similarly define $\fc_0$ via
    \<
        \label{eq:fc0}
        \fc_0 = -\Res_{s=\fr12} H(s) = -\tf12 \cdot \fr{\eta(\fr12)}{\xz(\fr12)} \Gamma(\tf12),    
    \>
noting the negative sign, and numerically compute that $\fc_0 \approx 1.28$. 

Formally shifting left \eqref{R:eq:phi0Reduc2} left to the line $\Re(s)=u$ for some $\tf14 < u < \tf12$, one has
    \<
        \label{R:eq:phi0Heuri}
        \phi_0(x,\xl) = -\fc_0 x^{1/4} + \sum_{\xr} \fc_\rho x^{\rho/2} + \vdint{u} \eta(s) \fr{\xz(2s)}{\xz(s)} \Gamma(\rho) x^{s/2} \,ds.
    \>
As $x^{\rho/2} = x^{1/4} e^{i\xg\log(x)/2}$, the sum $\sum_{\rho} \fc_\rho x^{\rho/2}$ is a sum of oscillatory (in $x$) terms, while the term $-\fc_0 x^{1/4}$ is strictly decreasing in $x$. Considering the first zero $\rho_1 \approx \nfr12 + 14.13i$ on the critical line, we find that $|\fc_{\rho_1}| \approx 1.27 \edot 10^{-9}$, which is far smaller than $\fc_0 \approx 1.28$; this vast difference is due to the exponential decay of $|\Gamma(s)|$ on vertical lines. It is conceivable then that the series $\sum_{\rho} |\fc_\rho|$ may converge, and may even satisfy the inequality $\sum_{\rho} |\fc_\rho| < \fc_0$. Should this inequality hold, it follows that
    \<
        \label{R:eq:phi0LessApprox}
        -\phi_0(x) \great x^{1/4}
    \>
and, from formulae \eqref{L:eq:pxxlForm} and \eqref{L:eq:pnxlpp*}, that 
    \[
        p(n,\xl) = (-1)^n \fr{\rho_*^{-n}\Phi(-\rho_*)}{\rt{2\pi \PsiD{2}(-\rho_*)}}\lh[1+O(n^{-1/5})\rh] = (-1)^n \expp{\tf12\xk\rt{n} + o(\rt{n})},
    \]
as mentioned in \eqref{L:eq:asm}. Thus, on assumptions of RH and SH we require detailed information about the quantities $\fc_\rho$.

\section{The cases of nonsimplicity and large residues}
\label{sec:RiemannNS}

Prior to a rigorous treatment of equation \eqref{R:eq:phi0Heuri}, we consider the scenario in which RH holds but SH fails. If a zero $\rho=\nfr12+i\xg$ has multiplicity $m>1$, then 
    \[
        \Res_{s=\rho} \lh[H(s)x^{s/2}\rh] = \sum_{k=1}^{m} \fc_{\rho,k} \fr{\log^{k-1}(x)}{(k-1)!} x^{\rho/2}
    \]
for some complex constants $\fc_{\rho,1},\ldots,\fc_{\rho,m}$ with $\fc_{\rho,m}\neq0$. In this case, the oscillatory sum in \eqref{R:eq:phi0Heuri} has terms of order larger than $x^{1/4}$, making the quantity $-\fc_0 x^{1/4}$ in equation \eqref{R:eq:phi0Heuri} negligible as $x \to \infty$. This is made precise in the following lemma.

\begin{lemma}
    \label{R:lem:NS}
    Assume RH and suppose that some nontrivial zero $\rho$ of $\xz(s)$ has multiplicity $m>1$. As $x\to\infty$ one has
        \[
            \phi_0(x,\xl) = \xO_{\pm}(x^{1/4}\log^{m-1} x).
        \]
\end{lemma}

\begin{proof}
    The result follows from a minor modification of the proof of Lemma \ref{M:lem:phi0Omega}, which is itself based on the proof of Theorem 15.3 in Montgomery and Vaughan \cite{montgomery2007multiplicative}, specifically following the discussions on page 467 therein.
\end{proof}

Considering Lemmata \ref{B:lem:DerivDecay}, \ref{B:lem:Omega*Approx}, and \ref{L:lem:phi0Omega} and Proposition \ref{B:prop:Biasymp}, it is evident that their results remain valid if expressions $x^{A-\xe}$ (where $A = \xQ/2$ for us) are changed to $x^{\nfr14}\log^{m-1}x$. As such, using the same arguments as in proving Theorem \ref{thm:nonRHBiasymp} we at once establish the biasymptotic result in Theorem \ref{thm:nonsimple}, namely: \emph{If RH holds but SH fails, then $\spnxl$ is biasymptotic with factor $\tf12$.}

We now introduce another hypothesis. We say that the Linear Independence Hypothesis (LI) holds if the set $\{|\xg| : \xz(1/2+i\xg) = 0\}$ is linearly independent over the set $\signs$. It is believed by some that said set of ordinates is even linearly independent over $\mathbb{Q}$, but this is more than needed here. The following arguments are taken from \cite{montgomery2007multiplicative}*{Thm.\ 15.5}, our arguments only require this weaker independence condition.

\begin{lemma}
    \label{lem:LI}
    Assume RH, SH, and LI, and suppose that there exist zeros $\rho_1, \rho_2, \ldots, \rho_K$ with ordinates $0 < \xg_1 < \xg_2 < \cdots < \xg_K$ such that 
        \<
            \label{R:eq:ResIneq}
            \sum_{k=1}^K |\fc_{\rho_k}| > \fc_0.
        \>
    Then writing $\phi_0(x,\xl)=\phi_0(x)$ one has
        \<
            \label{R:eq:phi0Sup}
            \limsup_{x\to\infty} \fr{\phi_0(x)}{x^{1/4}} \geq \tf12\sum_{k=1}^K |\fc_{\rho_k}| - \tf12\fc_0
        \>
    and
        \<
            \label{R:eq:phi0Inf}
            \liminf_{x\to\infty} \fr{\phi_0(x)}{x^{1/4}} \leq - \tf12\sum_{k=1}^K |\fc_{\rho_k}| - \tf12\fc_0, 
        \>
    and in particular one has
        \[
            \phi_0(x) = \xO_{\pm} (x^{1/4}).
        \]
\end{lemma}

\begin{proof}
    Let $c > 0$. Recalling that $\MelZero{\phi_0(x)}{s} = \int_0^1 \phi_0(x)x^{-s-1} \,ds$ and noting that
        \[
            \int_1^\infty \fr{cx^{1/4} - \phi_0(x)}{x^{s/2+1}} \,dx = \fr{2c}{s-1/2} - 2H(s) + \MelZero{\phi_0(x)}{-\nfr{s}{2}} \qquad (\xs>\tf12),
        \]
    we define 
        \begin{align}
            W\hspace{-0.08em}(s) &= \int_1^\infty \fr{cx^{1/4} - \phi_0(x)}{x^{s/2+1}} \,dx, \\
            \sW\hspace{-0.08em}(s;\vphi) &= \tf12 e^{i\vphi} W(s) + \tf12 e^{-i\vphi} W(\bar{s}),
        \end{align}
    and
        \<
            \label{R:eq:IDefin}
            I(s) = \int_1^{\infty} \fr{cx^{1/4}-\phi_0(x)}{x^{s/2+1}} \prod_{k=1}^K \lh\{ 1 + \cos(\vphi_k - \gamma_k \log x) \rh\} \,dx, 
        \>
    where $\vphi_1,\vphi_2,\ldots,\vphi_K$ are real numbers to be determined below. Note that $W\hspace{-0.08em}(s)$ has a simple pole at $s=1/2$ with residue $2c + \fc_0$, where $\fc_0 = -\Res_{s=1/2}[H(s)x^{s/2}]$. Multiplying out the product in \eqref{R:eq:IDefin} we have
        \<
            \label{R:eq:IFormula}
            I(s) = W\hspace{-0.08em}(s) + \sum_{k=1}^K \sW\hspace{-0.08em}(s + i \gamma_k ; \vphi_k) + J(s),
        \>
    where $J(s)$ is a linear combinations of terms of the form
        \[
            \sW\hspace{-0.08em}\lh(s + i \tssum_k \xe_k \gamma_k; \,\, \tssum_k \xe_k \vphi_k\rh)
        \]
    with all $\xe_k \in \signs$ and at least two of the $\xe_k$ nonzero.

    We now suppose that $\phi_0(x) \leq cx^{1/4}$ for all $x > 0$, noting that, by the statement of Landau's theorem (Lemma \ref{lem:Landau}), no generality is lost by assuming the inequality to hold for all $x$. Applying Landau's theorem we see that $I(s)$ converges for $\sigma > 1/2$, and the linear independence assumption implies that $J(s)$ is analytic at the point $s=1/2$. Thus, the function $I(s)$ has a simple pole at $s = 1/2$ with residue
        \[
            2c + \fc_0 + \sum_{k=1}^K \Re\lh[ e^{i\vphi_k} \fc_{\rho_k} \rh].
        \]
    
    Select the $\vphi_k$ so that $e^{i\vphi_k} \fc_k = -|\fc_k|$. As $I(s)$ is uniformly bounded below for $s > 1/2$, by letting $s$ tend to $1/2$ from above (on the real axis) in equation \eqref{R:eq:IFormula} we find that
        \[
            2c \geq \sum_{k=1}^K |\fc_{\rho_k}| - \fc_0,
        \]
    establishing inequality \eqref{R:eq:phi0Sup}. The proof of inequality \eqref{R:eq:phi0Inf} is similar, and the assertion that $\phi_0(x) = \xO_{\pm}(x^{1/4})$ follows at once from inequalities \eqref{R:eq:phi0Sup} and \eqref{R:eq:phi0Inf}.
\end{proof}

Since Lemma \ref{lem:LI} provides an $\xO$--result for $\phi_0(x)$ in place of, e.g., Lemma \ref{L:lem:phi0Omega}, we argue as in the proof of Theorem \ref{thm:nonRHBiasymp} to conclude the following similar result.

\begin{theorem}
    \label{R:thm:LargeResSubs}
    If the hypotheses of Lemma \ref{lem:LI} hold, then $\spnxl$ is biasymptotic with factor $\tf12$. 
\end{theorem}

\section{The case of small residues}
\label{sec:RiemannExact}

\subsection{A Quasi-exact Formula.}
We now consider conditions under which relation \eqref{L:eq:asm} holds, i.e., conditions on which
    \[
        p(n,\xl) = (-1)^n \expp{ \tf12 \xk\rt{n}+o(\rt{n}) } \qquad (\xk = \pi\rt{2/3},\,\, n \to \infty),
    \]
or equivalently conditions on which $((-1)^n p(n,\xl))_\nn$ is asymptotic with factor $\tf12$. Specifically, recalling again that
    \[
        \fz = \xz(2), \quad \xk=\pi\rt{2/3}, \quad\text{and}\quad \eta(s) = (\xk/\fz)^s(1-2^{-s-1})\xz(s+1),
    \]
we now conduct a more precise analysis of the integral
    \<
        \label{R:eq:phi0Restate}
        \phi_0(x) = \ldint{\xs} \eta(s)\fr{\xz(2s)}{\xz(s)}\Gamma(s)x^{s/2} \,ds \qquad (\xs > \tf12)
    \>
from \eqref{R:eq:phi0Heuri} by moving the line of integration to the left of $\xs = 1/2$. 

We first require two bounds for $\xz(s)$ on assumption of RH.

\begin{lemma}[{\cite{montgomery2007multiplicative}*{Thm.\ 13.22}}]
    \label{lem:ZetaTSeqMV}
    Assume RH. There exists an absolute constant $C$ such that for every $T \geq 4$ there is a $t$, with $T \leq t \leq T+1$, such that
        \[
            |\xz(\xs+it)| \geq \expp{-C\fr{\log{T}}{\log\log{T}} } \qquad \text{for $-1 \leq \xs \leq 2$.}
        \]
\end{lemma}

We make use of the above lemma via the following immediate corollary.

\begin{corollary}
    \label{R:cor:ZetaTSeq}
    Assume RH and let $\xe > 0$. For all $U > U_0=U_0(\xe)$ there exists a $T$, with $U \leq T \leq U+1$, such that
        \<
            \label{R:eq:ZetaTIneq}
            |\xz(\xs+iT)|^{-1} \leq U^\xe \qquad \text{for $-1 \leq \xs \leq 2$.}
        \> 
\end{corollary}

We now give a ``quasi-exact'' formula for $\phi_0(x)$ using \eqref{R:eq:phi0Restate}, where we use the descriptor quasi-exact since our arguments are in the spirit of Titchmarsh's exact formula for the Mertens function $M(x) = \sum_{n \leq x} \xm(n)$, see e.g.\ \cite{titchmarsh1986theory}*{Thm.\ 14.27}.

\begin{proposition}
    \label{R:prop:phi*Exact}
    Assume RH and SH. For all $x > x_0$ there exist quantities $T_x$, with $\logx \leq T_x \leq \logx+1$, such that
        \<
            \label{R:eq:QuasiExact}
            \phi_0(x) = \fc_0 x^{1/4} + \sum_{|\xg|<T_x} \fc_\rho x^{\rho/2} + O(x^{1/6}).
        \>
\end{proposition}

\begin{proof}
    Fix $\xe>0$. On assumption of RH, we recall from Lemma \ref{lem:ZetaPowerBB} that if $[\xs_0,\xs_1] \subset (\tf12,\infty)$ and $\xd > 0$, then $|\xz(\xs+it)|^{-1} \less t^\xd$ uniformly for $\xs_0 \leq \xs \leq \xs_1$. Letting $\xs_0 = 2/3$, $\xs_1 = 2$, and $\xd = 1$, it follows that there exist some $t_0 > 0$ and $c_0 > 0$ such that
        \<
            \label{R:eq:ZetaIneq}
            |\xz(\xs+it)|^{-1} \leq c_0 t \qquad (\tf23\leq\xs\leq 2,\,\,t>t_0).
        \>
    As this inequality clearly remains valid if $t_0$ is replaced with $\max\{t_0,1\}$, we assume without loss of generality that $t_0 \geq 1$. 

    Let $T > t_0 \geq 1$ be a large parameter to be determined below; these lower bounds allow for our use of inequality \eqref{R:eq:ZetaIneq} in the following arguments. Using equation \eqref{R:eq:phi0Restate} with $\xs = 1$ we truncate the integral at height $T$ to write
        \<
            \label{R:eq:phi0Trunc}
            \ldint{1} H(s)x^{s/2} \,ds = \dint_{1-iT}^{1+iT} \eta(s)\fr{\xz(2s)}{\xz(s)}\Gamma(s)x^{s/2} \,ds + \cR(x,T).
        \>
    As $\eta(s) \less 1$ uniformly for $1/3 \leq \xs \leq 1$ and for $t > T$ we have, by \eqref{eq:GammaBB}, that
        \[
            |\xz(1+it)|^{-1} \less t \qqand |\xG(1+it)| \less t^{1/2}e^{-\pi t/2},
        \]
    we deduce that
        \<
            \label{R:eq:phi0RBB}
            \cR(x,T) \less x^{1/2} \int_T^\infty t^{3/2} e^{-\pi t/2} \,dt \less x^{1/2} \int_T^\infty e^{-\pi t/3} \,dt \less x^{1/2} e^{-\pi T/3}.
        \>

    Returning to equation \eqref{R:eq:phi0Trunc} and integrating around the rectangle with vertices $1-iT$, $1+iT$, $1/3+iT$, and $1/3-iT$, the residue calculus provides that
        \<
            \label{R:eq:phi0Res}
            \begin{aligned}
                \dint_{1-iT}^{1+iT} \eta(s)\fr{\xz(2s)}{\xz(s)}\Gamma(s)x^{s/2} \,ds &= \lh[ \dint_{1-iT}^{\fr13-iT} + \,\dint_{\fr13-iT}^{\fr13+iT} + \,\dint_{\fr13+iT}^{1+iT} \rh] \eta(s)\fr{\xz(2s)}{\xz(s)}\Gamma(s)x^{s/2} \,ds \\
                & \qquad -\fc_0x^{1/4} + \sum_{|\xg|<T} \fc_\rho x^{\rho/2}.
            \end{aligned}
        \>
    First examining the integral from $\fr13+iT$ to $1+iT$ and again using \eqref{eq:GammaBB}, we have
        \<
            \label{R:eq:phi024} 
            \dint_{\fr13+iT}^{1+iT} \eta(s)\fr{\xz(2s)}{\xz(s)}\Gamma(s)x^{\nfr{s}{2}} \,ds \less x^{1/2} T^{1/2} e^{-\pi T/2} \int_{\fr13}^{1} \afrac{\xz(2\xs+2iT)}{\xz(\xs+iT)} \,d\xs.    
        \>
    As $\xz(s) \less |t|^{1/3+\xe}$ uniformly for $\xs \geq \fr23$ by Lemma \ref{lem:SSZeta}, we have
        \[
            |\xz(2\xs + 2iT)| \less T \qquad (\tf13 \leq \xs \leq 1),
        \] 
    so it remains to consider the factor $|\xz(\xs+iT)|^{-1}$ in \eqref{R:eq:phi024}. 
    
    For clarity in applying Corollary \ref{R:cor:ZetaTSeq}, using our fixed $\xe > 0$ let $U_0$ be the constant described in said corollary; by replacing $U_0$ with $\max\{U_0,t_0\}$ if necessary, we assume without loss of generality that $U_0 \geq t_0$. We now suppose that $U>U_0$ and that our $T$ from above is such that $U \leq T \leq U+1$ and 
        \[
            |\xz(\xs+iT)|^{-1} \leq U^\xe \qquad (-1\leq\xs\leq2).
        \]
    It then follows that
        \[
            x^{\nfr12} T^{\nfr12} e^{-\pi T/2} \int_{\fr13}^{1} \afrac{\xz(2\xs+2iT)}{\xz(\xs+iT)} \,d\xs
                \less x^{\nfr12} U^{\nfr12} e^{-\pi U/2} U^{1+\xe} \less x^{\nfr12} e^{-\pi U/3}.
        \]
    We deduce that
        \<
            \label{R:eq:phi024BB}
            \dint_{\fr13+iT}^{1+iT} \eta(s)\fr{\xz(2s)}{\xz(s)}\Gamma(s)x^{\nfr{s}{2}} \,ds \less x^{1/2} e^{-\pi U/3},
        \>
    and the same inequality also holds for the integral from $1-iT$ to $1/3-iT$ in \eqref{R:eq:phi0Res}.

    For the integral from $1/3-iT$ to $1/3+iT$ in \eqref{R:eq:phi0Res} we use the functional equation
        \[
            \xz(1-s) = 2^{1-s}\pi^{-s}\cos(\tf12\pi s)\Gamma(s)\xz(s)
        \]
    to write
        \<
            \label{R:eq:phi03}
            \dint_{\fr13-iT}^{\fr13+iT} \eta(s) \fr{\xz(2s)}{\xz(s)}\Gamma(s)x^{s/2} \,ds 
            = \dint_{\fr13-iT}^{\fr13+iT} \fr{2\eta(s)\xz(2s)}{(2\pi)^s\xz(1-s)} \cos(\tf12\pi s) \Gamma^2(s) x^{s/2} \, ds.
        \>
    For $s = 1/3 + it$ with $t > 0$ we bound
        \[
            |\Gamma(s)|^2 \less ( \tf19 + t^2)^{-1/6}e^{-\pi t} \qqand (2\pi)^{-s}\eta(s)\cos(\tf12\pi s) \less e^{\pi t/2},
        \] 
    whereby integral \eqref{R:eq:phi03} is
        \[
            \less x^{1/6} \int_{0}^T \ABS{ \fr{\xz(\fr23+2it)}{\xz(\fr23-it)}} (\tf19+t^2)^{-1/6} e^{-\pi t/2} \,dt.
        \]
    We first change the lower limit of integration here from $0$ to $t_0$ at a cost of $O(x^{1/6})$. Then using arguments and bounds similar to those above, it follows that
        \<
            \label{R:eq:phi03BB}
            \dint_{\fr13-iT}^{\fr13+iT} \eta(s) \fr{\xz(2s)}{\xz(s)}\Gamma(s)x^{s/2} \,ds \less x^{\nfr16} \int_{t_0}^\infty t^{2-\nfr23}e^{-\pi t/2} \,dt \less x^{\nfr16}.
        \>

    Combining statements \eqref{R:eq:phi0Trunc}--\eqref{R:eq:phi0Res}, \eqref{R:eq:phi024BB}, and \eqref{R:eq:phi03BB} we deduce that
        \[
            \ldint{1} H(s)x^{s/2} \,ds = -\fc_0 x^{1/4} + \sum_{|\xg|<T} \fc_\rho x^{\rho/2} + O(x^{1/2}e^{-\pi U/3}) + O(x^{1/6}).
        \]
    Recalling that $T \in [U,U+1]$ depends on $U$, we let $U = \log x$ and assume that $x$ is sufficiently large that $\log x > U_0$, whereby $x^{1/2} e^{-\pi U/3} = x^{1/2-\pi/3} \less x^{-1/2}$ and the result follows.
\end{proof}

Recalling from equation \eqref{R:eq:phi*Reduc} that $\phi_*(x) = \phi_0(x) + O(x^\xe)$, we have the following immediate corollary.

\begin{corollary}
    \label{R:cor:phi*Asymp}
    Assume RH and SH. For all $x > x_0$ there exist quantities $T_x$, with $\logx \leq T_x \leq \logx+1$, such that
        \[
            \phi_*(x) = -\fc_0 x^{1/4} + \sum_{|\xg|<T_x} \fc_\rho x^{\rho/2} + O(x^{1/6}).
        \]
\end{corollary}

\subsection{The sum \tops{$\sum_\rho |\fc_\rho|$}.}
As discussed in section \ref{sec:Riemann}, specifically following equation \eqref{R:eq:phi0Heuri}, if (on RH and SH) one has $\sum_{\rho} |\fc_\rho| < \fc_0$, where
    \[
        \fc_\rho = \big[ \eta(\rho) \big] \fr{\xz(2\rho)}{\xz'(\rho)}\Gamma(\rho) = \Big[ (\xk/\fz)^{\rho}(1-2^{-\rho-1}) \xz(\rho+1) \Big] \fr{\xz(2\rho)}{\xz'(\rho)} \Gamma(\rho)
    \] 
and
    \[
        \fc_0 = -\tf12 \cdot \fr{\eta(\fr12)}{\xz(\fr12)} \Gamma(\tf12) \approx 1.28,
    \]
then the sequence $((-1)^n p(n,\xl))_\nn$ is asymptotic with factor $1/2$.

In addition to RH and SH, we require an assumption of Effective Simplicity (ES), namely that (on SH) one has $|1/\xz'(\rho)| \leq C |\rho|$ for some $C>0$ and all nontrivial zeros $\xr = 1/2+i\xg$ of $\xz(s)$. Toward establishing upper bounds for the quantities $|\fc_\rho|$, we record a result due to Trudgian.

\begin{lemma}[{\cite{trudgian2014new}*{Thm.\ 1.1}}]
	\label{lem:TrudgianBB}
	When $t \geq 3$, one has $|\xz(1+it)| \leq \tf34 \log t$. 
\end{lemma}

Because $|\fc_\rho| = |\fc_{\bar\rho}|$ for all zeros $\rho = 1/2+i\xg$, it suffices to consider only those $\rho$ with $\xg > 0$ in the following lemma. 

\begin{lemma}
    \label{lem:CritCoeffBB}
	Assume that RH, SH, and ES all hold, the latter with constant $C$. For all nontrivial zeros $\rho=1/2+i\xg$ of $\xz(s)$ with $\xg > 0$, one has
		\[
            |\fc_\rho| < 31 C \log(\xg)\exp(-\tfr{\pi}{2}\xg).
		\]
\end{lemma}

\begin{proof}
    As the ordinates $\xg$ of the zeros $1/2 + i\xg$ are all $> 14$, Trudgian's bound \ref{lem:TrudgianBB} implies that $|\xz(1+2i\xg)| \leq \log(2\xg) < 2\log{\xg}$. By virtue of the identity 
        \[
            \Gamma(\tf12+it)\Gamma(\tf12-it) = \fr{2\pi}{\cosh(\pi t)},
        \]
    for all real $t$ one has $|\Gamma(\tf12+it)| \leq \rt{2\pi} \exp(-\tfr{\pi}{2}|t|)$. Combining these bounds with inequality ES and the crude numerical estimate
		\[
			|\xh(\xr)| = \lh| (2\rt{6}/\pi)^\rho (1-2^{-\rho-1}) \xz(\rho+1) \rh| \leq (2\rt{6}/\pi)^{\nfr12} \xz(\nfr32) |1-2^{-\nfr32-i\xg}| < 6,
		\]
	we have
		\[
			|\fc_\rho| < 12 \rt{2\pi} \,C \log(\xg) \exp(-\tfr{\pi}{2}\xg) < 31 C \log(\xg) \exp(-\tfr{\pi}{2}\xg),
		\]
	as claimed.
\end{proof}

The following result, also due to Trudgian, allows us to explicitly bound the number of zeros $\rho=\tf12+i\xg$ with ordinates $\xg$ in a given interval $[2^k,2^{k+1}]$ with $k \geq 2$. For $T > 0$ let
    \[
        N(T) := \#\{\xg : \xz(\tf12+i\xg)=0,\,\, 0< \xg \leq T\}.     
    \]

\begin{lemma}[{\cite{trudgian2014improved}*{Cor.\ 1.1}}]
	For all $T \geq T_0 \geq e$, one has
		\<
			\label{eq:TrudgianBBN}
			\lh| N(T) - \fr{T}{2\pi} \log\lh( \fr{T}{2\pi e} \rh) - \fr78 \rh| \leq c_1 \log T + c_2 \log\log T + c_3 + \fr{c_4}{T},
		\>
	where $c_1 = 0.111$, $c_2 = 0.275$, $c_3 = 2.45$, and $c_4 = 0.2$.
\end{lemma}

The following lemma helps to give a crude bound on the tail sums of $\sum_{\xg>0} |\fc_\xr|$.
\begin{lemma}
    \label{lem:STailBB}
	For all $K>5$ one has
		\<
            \label{eq:STailBB}
			\sum_{\xg > 2^K} \log(\xg)\exp(-\tf{\pi}{2}\xg) \leq \sum_{k=1}^\infty \exp(-\tf{\pi}{5} 2^{K+j}).
		\>
\end{lemma}

\begin{proof}
	We first establish that for $k > 5$ one has the crude estimate
		\<
			\label{eq:ZetaNBB}
			N(2^{k+1}) - N(2^k) \leq (k+1) 2^{k}.
		\>
	Letting $R(T)$ denote the right side of inequality \eqref{eq:TrudgianBBN}, one has
		\begin{align*}
				N(2T) - N(T) &\leq \fr{2T}{2\pi} \log\lh(\fr{2 T}{2 \pi e}\rh) - \fr{T}{2\pi}\log\lh( \fr{T}{2\pi e}\rh) + R(2T) + R(T) \\
                & \leq \fr{T}{2\pi} \log\lh( \fr{4 T}{2 \pi e} \rh) + 2 R(2T).			
		\end{align*}
	Setting $T = 2^k$, it follows that
		\[
			N(2^{k+1}) - N(2^k) \leq \fr{2^k}{2\pi} (k+1) + 2 \lh( c_1 (k+1) + c_2 \log(k+1) + c_3 + \fr{c_4}{e} \rh),
		\]
	and we directly check this latter quantity to be at most $(k+1) 2^k$ for $k \geq 5$.
	
	The terms $\log(\xg)\exp(-\tfr{\pi}{2}\xg)$ are decreasing in $\xg$ when $\xg > 3$, so that using inequality \eqref{eq:ZetaNBB} the left side of \eqref{eq:STailBB} is
		\<
			\label{eq:STailBB2}
			= \sum_{k = K}^\infty \,\,\sum_{2^k < \xg \leq 2^{k+1}} \log(\xg)\exp(-\tfr{\pi}{2}\xg)	
			\leq \sum_{k = K}^\infty (k+1) 2^k \cdot \log(2^k)\exp(-\tfr{\pi}{2}2^k).
		\>
	Verifying that
		\[
			(K+1)2^K \log(2^K) \expp{-\rfr{\pi}{2} 2^K} \leq \expp{-\tf25\pi 2^K} \qquad \text{for $K > 5$,}
		\]
	the right side of \eqref{eq:STailBB2} is
		\[
			\leq \sum_{k = K}^\infty \expp{-\tf25 \pi 2^k}
			= \sum_{j=1}^\infty \expp{- \tf15\pi 2^{K+j}}
		\]
	as claimed.
\end{proof}

Thus, assuming RH, SH, and ES to all hold, the latter with constant $C$, we apply Lemmata \ref{lem:CritCoeffBB} and \ref{lem:STailBB} to see that
    \<
        \label{eq:STail}
        \sum_{\xg > 2^K} |\fc_\rho| \leq 31\,C \sum_{\xg > 2^K}\log(\xg)\exp(-\tfr{\pi}{2}\xg) \leq 31\,C  \sum_{j=1}^\infty \expp{-\tfr{\pi}{5} 2^{K+j}}.        
    \>
Setting $K=15$, we numerically compute that 
    \<
        \label{eq:K15}
        \sum_{\xg \leq 2^{15}} |\fc_\xr| \approx 1.27 \edot 10^{-9} \qqand \sum_{\xg > 2^{15}} |\fc_\xr| \approx 2.18\,C \edot 10^{-17882}.
    \>
Rounding up the quantities here, it follows that
    \<
        \label{eq:ScrudeBB}
        \sum_{\rho} |\fc_\rho| \leq 10^{-8} + 10^{-17881} C,
    \>
and we use this bound to at last establish the formal version of statement \ref{it:asm}, namely Theorem \ref{thm:RAsymp}, which we restate here for convenience.

\begin{theorem*}[Thm.\ \ref{thm:RAsymp}]
    Suppose that RH, SH, and ES hold, the latter with constant $C > 0$. There exists a constant $\fc > 0$ such that: if the ES constant $C < \fc$, then
        \<
            (-1)^n \logsc p(n,\xl) \sim \tf12 \xk\rt{n},
        \>
    namely $((-1)^n p(n,\xl))_\nn$ is asymptotic with factor $\fr12$. Moreover, it holds that 
    {\binoppenalty=10000\relpenalty=10000 $\fc > 10^{17881}$}.
\end{theorem*}

\begin{proof}
    By inequality \eqref{eq:ScrudeBB}, one has
        \[
		    \fc_0 + \sum_{\rho} |\fc_\rho| \leq -1.28 + 10^{-8} + 10^{-17881}\,C \leq -2 + 10^{-17881}\,C.
	    \]
Thus if $C < 10^{17881}$, then by Corollary \ref{R:cor:phi*Asymp} one has 
	\[
		\phi_*(x) \leq -c_1x^{1/4} \qquad (x\to\infty)
	\]
for some $c_1 > 0$, whereby
    \begin{align}
        p(n,\xl) &= e^{\phi(n)} \lf[ e^{\phi_*(n)} + (-1)^n e^{-\phi_*(n)} + O(n^{-1/5}e^{|\phi_*(n)|}) \rh] \notag \\ 
        \label{R:eq:pxxlApprox}  
        &= (-1)^n e^{\phi(n) - \phi_*(n)} \lf[ 1+O(n^{-1/5}) \rh]
    \end{align}
as $n\to\infty$. The result then follows, as Proposition \ref{L:prop:pp*Asymp} and equation \eqref{R:eq:pxxlApprox} imply that
    \begin{align*}
        (-1)^n \logsc p(n,\xl) = \phi(n) - \phi_*(n) + O(n^{-1/5}) = \tf12\xk\rt{n} + O(n^{1/4+\xe}).
    \end{align*}
\end{proof}

\begin{remark}
    The choice $K=15$ in \eqref{eq:K15} is made for the relative ease of computation. Running Mathematica on a modern processor using eight cores in parallel, the computation of $\sum_{\xg \leq 2^{15}} |\fc_\rho|$ in equation \eqref{eq:K15} involved the first $39,424$ zeros of $\xz(s)$ on the critical line and took approximately thirty minutes to complete. Considering that billions of zeros of $\xz(s)$ on the critical line have been computed, with greater time and resources larger $K$ may be used, and since the summands in the majorant in \eqref{eq:STail} are doubly-exponentially decaying, any increases in $K$ will drastically increase the lower bound for the constant $\fc$ in Theorem \ref{thm:RAsymp}.
\end{remark}

\section{\tops{Explicit estimates for $\phi_*(x,\xm)$.}}
\label{sec:explicit}

We conclude our analyses with an informal discussion of an interesting problem concerning the sequence $\spnmu$.

Theorem \ref{thm:muBiasymp} states that if $\xQ<1$, then $\spnmu$ is biasymptotic with factor $\xk^{-1} = \pi^{-1}\rt{3/2}$. Recalling from section \ref{sec:Biasymp} that $p(n,\xm) \approx (-1)^n \exp(\rt{n})$ when $50 \leq n \leq 10^5$, it is natural to ponder: Should it hold that $\xQ < 1$, how large must $n$ be before one has some nontrivial (i.e., of length greater than 1) string $(N,N+1,\ldots,N+L-1)$ on which $p(n,\xm) \approx \exp(\rt{n})$? For simplicity, in this section we assume RH, SH, and ES with constant $C$ (although this $C$ will not be relevant) and estimate an answer to this question.

As seen in sections \ref{sec:OmegaPrelim} and \ref{sec:MBiasymp}, the locations of such strings are determined by $\phi_*(x,\xm)$. By Proposition \ref{M:prop:pp*Asymp}, on RH (recall that $\tf{\xQ}{2} = \xvq = \tf14$ in this case) we have 
    \<
        \label{R:eq:Muphi*AsympRecall}
        \phi_*(x,\xm) = \phi_0(x,\xm) + O(x^{\xe}),  
    \>
where
    \[
        \phi_0(x,\xm) = \vdint{\xs} F(s)x^{s/2} \,ds = \vdint{\xs} 2^s \pfrac{1-2^{-s-1}}{1-2^{-s}} \fr{\xz(s+1)}{\xz(s)} \Gamma(s)x^{s/2} \,ds \qquad (\xs > \tf12).
    \]

We first note that $F(s)x^{s/2}$ has a triple pole at $s=0$ with residue
    \[
        \begin{aligned}
            \Res_{s=0}[F(s)x^{s/2}] &= - \fr{\log^2 x}{8\ell_2} - \pfrac{3\ell_2-2\ell_\pi}{4\ell_2} \log x - \fr{\pi^2 + 2(\ell_2^2 - 18\ell_2\ell_\pi + 6\ell_\pi^2)}{24\ell_2} \\
            &\approx -0.18 \log^2 x + 0.076 \log x + 0.12,
        \end{aligned}
    \]
where $\ell_2 = \log 2$ and $\ell_\pi = \log \pi$.
Making arguments similar to those of Proposition \ref{R:prop:phi*Exact}, we find that
    \<
        \label{R:eq:Muphi*Approx}
        \phi_0(x,\xm) = \sum_{|\xg| < T_x} \fc_\rho x^{\rho/2} - \fr{\log^2 x}{8\log 2} + O(\log x),
    \>
where $\fc_\rho := \Res_{s=\rho} F(s)$ and $T_x \asymp \logx$.

For each zero $\rho = 1/2 + i\xg$ with $\xg>0$, the sum in equation \eqref{R:eq:Muphi*Approx} includes the sum
    \[
        \fc_\rho x^{\rho/2} + \fc_{\bar\rho} x^{\bar{\rho}/2} = 2 \Re[\fc_\rho x^{i \xg/2}] x^{1/4}.
    \]
Let $\rho_1$ and $\rho_2$ be the first two zeros of $\xz(s)$ on the critical line; numerically one has
    \begin{alignat*}{2}
        &\rho_1 \approx \tf12 + 14.13i, \qquad && |\fc_{\rho_1}| \approx 4.41 \edot 10^{-10}, \\
        &\rho_2 \approx \tf12 + 21.02i, \qquad && |\fc_{\rho_2}| \approx 7.78 \edot 10^{-15}.
    \end{alignat*}
The large difference between $|\fc_{\rho_1}|$ and $|\fc_{\rho_2}|$ is largely due to the rapid decay of $|\Gamma(s)|$ along vertical lines. 
Thus, assuming that the residues of $1/\xz(s)$ at zeros $\xr = 1/2+i\xg$ do not grow too rapidly, it is reasonable to approximate the sum $\sum_{\rho} \fc_\rho x^{\rho/2}$ using only the terms contributed by $\rho_1$.

Setting $\xg_1 := \Im \rho_1$, $\fa_1 := \Re \fc_1$, and $\fb_1 := \Im \fc_1$, one has
    \[
        \fc_{\rho_1} = \fa_1+i\fb_1 \approx -4.7 \edot 10^{-11} + 4.4 \edot 10^{-10}i
    \]
and
    \[
        2 \Re[\fc_{\rho_1} x^{i \xg_1/2}]x^{1/4} = \big( 2\fa_1 \cos(\tf12\xg_1\logx) - 2\fb_1 \sin(\tf12\xg_1 \logx) \big) x^{1/4}.
    \]
Moreover, as $|\fb_1| \approx 10|\fa_1|$, we may further say that
    \<
        \label{R:eq:rho1Summand}
        2 \Re[\fc_1 x^{i \xg_1 /2}]x^{1/4} \approx -2 \fb_1 \sin(\tf12\xg_1 \log x) x^{1/4}.
    \>

Considering equations \eqref{R:eq:Muphi*AsympRecall}--\eqref{R:eq:rho1Summand}, it is thus reasonable to approximate
    \<
        \label{R:eq:phi*Approx}
        \phi_*(x,\xm) \approx -0.18 \log^2 x - 2 \fb_1 \sin(\tf12 \xg_1\log x) x^{1/4}.
    \>
With this approximation, one sees that although the $x^{1/4}$ term eventually dominates the $\log^2 x$ term here, this trend is not detectable until the quantity $x^{1/4}$ has overcome the significant size difference between $|\fc_0|$ and $|\fb_1|$. We include plots of the approximation \eqref{R:eq:phi*Approx} for $\phi_*(x,\xm)$ in Figure \ref{fig:phi*Plot} below.

Considering the plots in Figure \ref{fig:phi*Plot}, we do not expect $\phi_*(x,\xm)$ to take any positive values until $x \approx 1.2 \edot 10^{50}$, and therefore we do not expect that $p(n,\xm) \approx A e^{\rt{n}}$ (for some $A>0$) on any nontrivial strings $(N,N+1,\ldots,N+L-1)$ until $N$ is larger than $10^{50}$, assuming that such strings even exist. Moreover, the steady, negative trend displayed in the upper plot may explain why the sequence $\spnmu$ does not exhibit any biasymptotic behavior for small $n$.

One may ``efficiently'' compute the numbers $p(n,1)$ recursively \cite{apostol1976introduction}*{Thm.\ 14.8} via 
    \[
        p(n,1) = \fr{1}{n} \sum_{k=1}^{n-1} p(k,1) \xs(n-k),
    \] 
where $\xs(m) := \sum_{d\;|\;m} d$. For general $f:\nn \to \signs$ one has the similar identity
    \<
        \label{R:eq:pnfRecur}
        p(n,f) = \fr{1}{n} \sum_{k=1}^{n-1} p(k,f) \fS(n-k,f),    
    \> 
where 
    \[
        \fS(m) := \sum_{d \;|\; m} d \cdot\! f(d)^{m/d}.    
    \]

Using identity \eqref{R:eq:pnfRecur}, a 2020-model personal laptop, and the C-based language GP/Pari, the computation of $p(n,\xm)$ for $n \leq 10^5$ took approximately 11 minutes and resulted in a file (to store the values) 10.6 megabytes in size. These time and memory requirements are both roughly thirty-fold increases of the requirements for computing $p(n,f)$ for $n \leq 10^4$; these latter requirements are themselves roughly thirty-fold increases of those for $n \leq 10^3$. 

Thus, it is clear that computing $p(n,f)$ for large $n$ requires vastly superior methods and resources than those of the author. 
{
\begin{figure}[ht!]
\centering
\includegraphics[width=0.8\textwidth]{"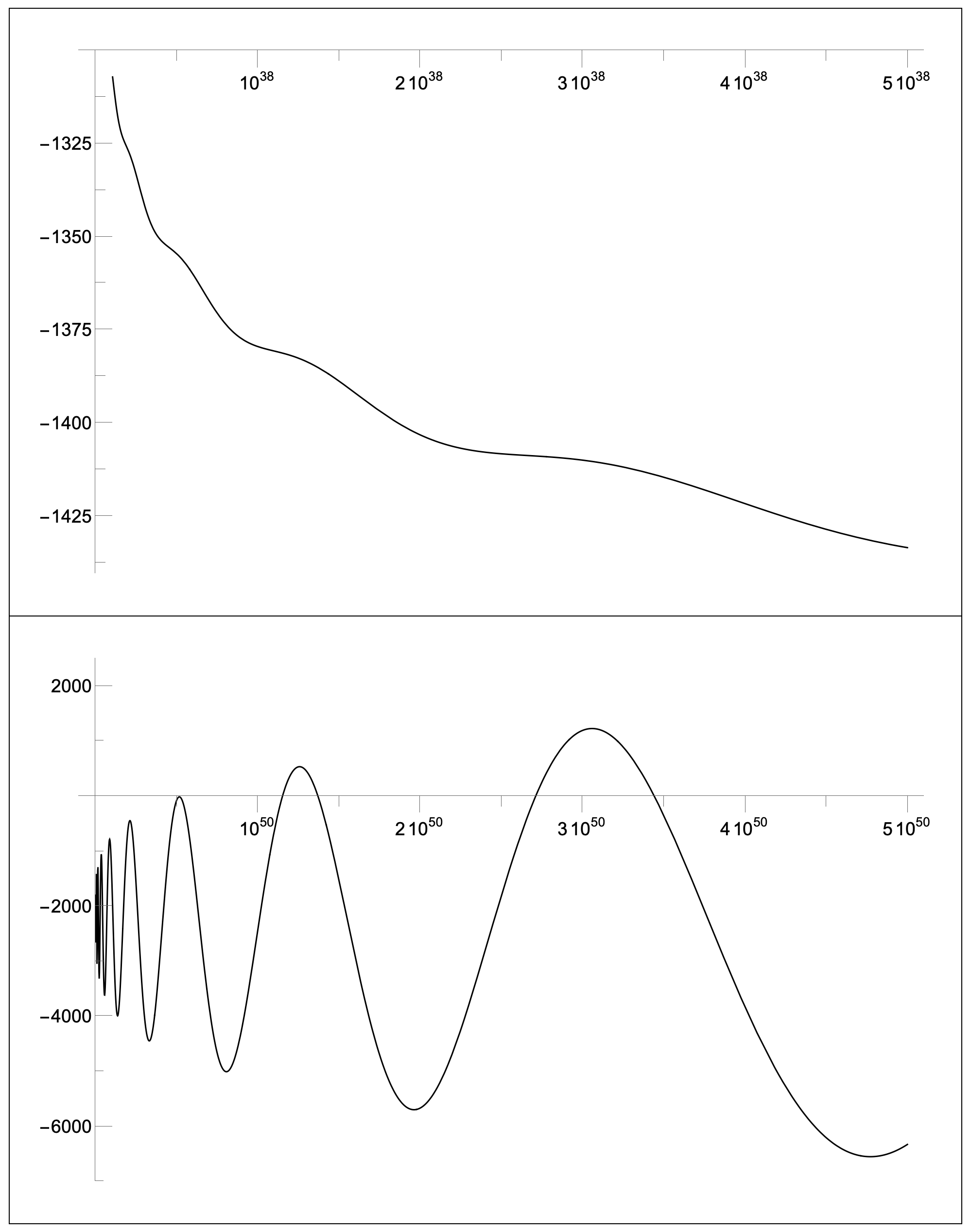"}
\caption{Plots of the right side of \eqref{R:eq:phi*Approx} on two different input ranges.}
\label{fig:phi*Plot}
\end{figure}
}

\subsection{Zero values.}
As mentioned in section \ref{sec:Biasymp}, item \ref{it:zero}, there are some $n$ for which $p(n,\mu)=0$. In particular, for $n \leq 10^5$ one has $p(n,\xm)=0$ only for $n$ equal to $2$, $4$, $5$, $7$, $8$, $11$, and $15$. It is natural then to ask if there exist any other $n$ for which $p(n,\xm)=0$. Recalling the formula
    \<
        \label{R:eq:pxmuRecall}
        p(n,\xm) = e^{\phi(n)} \lh[ e^{\phi_*(n)} + (-1)^{n} e^{-\phi_*(n)}+ \O{n^{-1/5}e^{|\phi_*(n)|}} \rh]
    \> 
and considering the approximation \eqref{R:eq:phi*Approx}, it is reasonable that $n$ for which $p(n,\mu)=0$ may exist when $\phi_*(n,\xm) \approx 0$. However, as the discussion on Figures shows, such $n$ are very likely to be in ranges requiring significant computational resources. By equation \eqref{R:eq:pxmuRecall} and the relations \eqref{eq:logpn}, any such $n$ must necessarily be odd, excepting the values $2$, $4$, and $8$ from the previous paragraph.

We may similarly analyze $p(n,\xl)$ and $\phi_*(n,\xl)$, but of course it only reasonable to do so if we believe that $\spnxl$ is indeed biasymptotic. As seen in sections \ref{sec:Riemann}--\ref{sec:RiemannExact}, if the Riemann Hypothesis is true, as many believe it to be, then the asymptotic or biasymptotic behavior of $\spnxl$ is uncertain. Moreover, given the lower bound on the constant $\fc$ in Theorem \ref{thm:RAsymp}, which can be drastically improved with additional computation, it is not unreasonable to think that $((-1)^n p(n,\xl))_\nn$ may be asymptotic. 

Finally, for the curious reader we note that for $n \leq 10^5$, one has $p(n,\xl)=0$ only when $n$ is 2, 8, 9, 11, or 25.

\addappendix

Here we follow arguments given for \cite{titchmarsh1986theory}*{Thm.\ 14.2} to give a proof of the generalized version of Littlewood's result that: On RH, one has
    \[
        \log\xz(s) \less (\log t)^{2-2\xs+\xe}    
    \]
uniformly for $1/2 < \xs_0 \leq \xs \leq 1$. 

\begin{lemma}
    \label{A:Littlewood}
	Suppose that $\xQ < 1$. One has
		\<
			|\log\xz(s)| \less (\log t)^{\fr{1-\xs}{1-\xQ}+\xe}
		\>
	uniformly for $\xQ < \xs_0 \leq \xs \leq 1$. 
\end{lemma}

For convenience we break the proof of Lemma \ref{A:Littlewood} into two pieces. We note that in the following two proofs, quantities $c_0,c_1,c_2,\ldots$ are universal positive constants.

\begin{lemma}
	\label{lem:lzDisk}
	Suppose that $\xQ < 1$. There exists $c > 0$ such that for all $t>t_0$ and all sufficiently small $\xd>0$ one has
		\[
			\lh|\log\xz\big((2+it)+w\big)\rh| \leq \fr{c}{\xd} \log{t}
		\]
	for all $|w| \leq 2-\xQ-\xd$. 
\end{lemma}

\begin{proof}
    For all $\xa \in (0,1)$ one has \cite{tenenbaum2015introduction}*{Thm.\ 3.9} that
        \<
            \label{eq:ZetaTenenbaum}
            |\xz(s)| \leq \fr{3t^{1-\xa}}{2\xa(1-\xa)} \qquad \text{for $\xs \geq \xa$ and $t \geq 1$.}
        \>
    Restricting our attention to only those $\xa$ for which $\xQ<\xa<1$, we restate \eqref{eq:ZetaTenenbaum} as
        \<
            \label{eq:ZetaTenenbaumBB}
            \log|\xz(s)| \leq (1-\xa)\log{t} + \log\pfrac{3/2}{\xa(1-\xa)} \qquad \text{for $\xs \geq \xa$ and $t \geq 1$.}
        \> 
    Let $0 < \xd \leq 1-\xQ$ so that $\xQ + \tf12\xd < 1$. Setting $\xa := \xQ + \tf12\xd$ and noting that $\fr{1}{\xa(1-\xa)}$ is increasing as a function of $\xa$ on $(\tf12,1)$, our assumption that $\xd \leq 1-\xQ$ implies that
        \<
			\label{eq:lzCrudeBB}
			\log |\xz(s)| \leq (1-\xQ-\tf12\xd) \log t + c_0 \qquad (\xs \geq \xQ+\tf12\xd,\,\, t \geq 1).
		\>
	
	Let $t > 3$ and fix $s=2+it$. By assumptions on $\xQ$ and $t$ the function $w \mapsto \log\xz(s+w)$ is analytic on the disk $|w| < 2-\xQ$. Moreover, since $\Im(s+w) > \tf32$ for all $|w| < 2-\xQ$, we infer from inequality \eqref{eq:lzCrudeBB} that
		\<
            \label{eq:BorelCarathBB}
			\Re\!\big[\.\log\xz(s+w)\big] \leq c_1\log(t+w) \leq c_2\log{t}
		\>
	for $|w| \leq 2 - \xQ - \tf12\xd$.

    Using \eqref{eq:BorelCarathBB} and applying the Borel-Carath\'eodory theorem \cite{lang1999complex}*{Ch.\ XII Thm.\ 3.1} to the function $\log\xz(s+w)$ on the circles with radii $R = 2-\xQ-\tf12\xd$ and $r = 2-\xQ-\xd$, we find that
		\begin{align*}
			|\log\xz(s+w)| 
			& \leq \fr{2(2-\xQ-\xd)(c_2\log{t})}{\tf12\xd} + \fr{4-2\xQ-\fr32\xd}{\tf12\xd} |\log\xz(s)| \\
			& = \lh(\fr{c_3}{\xd}-c_4\rh)\log{t} +\lh(\fr{c_5}{\xd}-c_6\rh) |\log\xz(s)| \\
            & \leq \fr1{\xd}(c_7 \log{t} + c_8)
        \end{align*}
	for $|w| \leq 2-\xQ-\xd$, and the result follows.
\end{proof}

\begin{proof}[Proof of Lemma \ref{A:Littlewood}]
    Fixing $\xe>0$, we show that there exist positive $c_\xe$ and $t_0$ such that
		\<
			\label{eq:lzCrudeBB2}
			|\log\xz(s)| \leq c_\xe (\log{t})^{\ellexp + \xe} \qquad \lh( t > t_0,\,\, \xQ + \fr{1}{\log\log{t}} \leq \xs \leq 1 \rh),
		\>
    from which the result follows for all $t$ large enough that $\xQ + (\log\log{t})^{-1} \leq \xs_0$.

	Let $t$ be sufficiently large, fix $\xs_*$ with $1 < \xs_* \leq \tf12t$, and let $0 < \xd < \min\{1-\xQ,\xs_*-1\}$. By our assumptions, the function $w \mapsto \log\xz(\xs_* + it + w)$ is analytic for $|w| < \xs_* - \xQ$. 
	For $\xs$ with $\xQ + \xd \leq \xs \leq 1$ we aim to bound $|\log\xz(\xs+it)|$ using Hadamard's three circles lemma \cite{tenenbaum2015introduction}*{Lem.\ 4.3} with circles centered at $\xs_*+it$ and passing through the points $1+\xd+it$, $\xs+it$, and $\xQ+\xd+it$, respectively. The radii of three circles are thus
		\[
			r_1:=\xs_*-1-\xd, \quad r:=\xs_*-\xs, \quad\text{and}\quad r_2:=\xs_*-\xQ-\xd;
		\]
	we include Figure \ref{fig:ThreeCircles} as a convenient reference. For $r_1 \leq r \leq r_2$ define
        \[
			M_r = \max_{|w|=r} \,\lh|\log\xz \big( (\xs_*+it) + w\big) \rh|
		\]
	and let $M_1 = M_{r_1}$ and $M_2 = M_{r_2}$.

    {
    \begin{figure}[ht!]
    \centering
    \begin{tikzpicture}[scale=0.5]
        \def\zMt{0.5} 
        \def\zmt{0.3} 
        \def\zu{10} 
        \def\zT{0.1} 
        \def\zss{0.6} 
        \def\zd{0.1} 
        \def\zs0{1.7} 
        \def\zang{15}
      
        \draw[thick] (-0.2*\zu,0)--(2.2*\zu,0); 
        \draw[] (\zu, 0.75*\zMt) -- (\zu, -2*\zMt) node[below] {$1$};
        \draw[] (0, 0.75*\zMt) -- (0, -2*\zMt) node[below] {$\xQ$};
        
        \node (s0) at (\zs0*\zu,0) {};
        \draw[] ($(s0)+(0,\zMt)$) -- ($(s0)+(0,-\zMt)$) node[below] {$\vphantom{\xQ}\hphantom{{}_*}\sigma_*$};
        \draw[thin] ($(s0)+(0,\zMt)$) -- ($(s0)+(0,-\zMt)$);
      
        \centerarc[red](s0)(180+5 : 180-5 : {(\zs0-\zd-0.1)*\zu}); 
        \centerarc[red](s0)(180+6 : 180-6 : {(\zs0-\zss)*\zu}); 
        \centerarc[red](s0)(180+9 : 180-9 : {(\zs0-1-\zd)*\zu}); 
      
        \node (Td) at (2*\zd*\zu,0) {};
        \node[below right] at (2*\zd*\zu,0) {$\xQ+\xd$};
        \draw[] ($(Td)+(0,\zMt)$) -- ($(Td)+(0,-\zMt)$);
      
        \node (ss) at (\zss*\zu,0) {};
        \node[below right] at (\zss*\zu,0) {$\vphantom{\xQ}\sigma$};
        \draw[] ($(ss)+(0,\zMt)$) -- ($(ss)+(0,-\zMt)$);
      
        \node (d1) at (1.1*\zu,0) {};
        \node[below right] at (1.1*\zu,0) {$\vphantom{\xQ}1+\xd$};
        \draw[] ($(d1)+(0,\zMt)$) -- ($(d1)+(0,-\zMt)$);
    \end{tikzpicture}
    \caption{Small pieces of the three circles.}
    \label{fig:ThreeCircles}
    \end{figure}
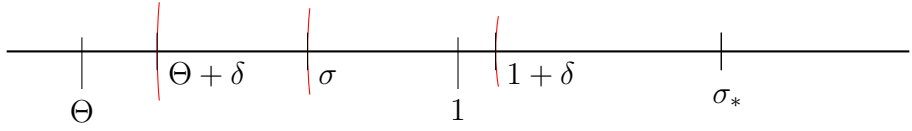
    }
        
	Since
		\[
			\log\xz(s) = \sum_{p\,\text{prime}} \sum_{k=1}^\infty \fr{1}{k p^{ks}} \qquad (\xs > 1),
		\]
	one has
		\<
            \label{eq:lzM1BB}
			M_1 \leq \sum_{p\,\text{prime}} \sum_{k=1}^\infty \fr{1}{k p^{k(1+\xd)}} < \sum_{n=1}^\infty \fr{1}{n^{1+\xd}} \leq \fr{c_9}{\xd},
		\>
    and using Lemma \ref{lem:lzDisk} it follows that $M_2 \leq c_{10}\xd^{-1}\log{t}$.
    
    Using Hadamard's lemma we bound $M_r \leq M_1^{1-\xl} M_2^{\xl}$ for $r_1 \leq r \leq r_2$, where	
		\[
			\xl 
				= \fr{\log(r/r_1)}{\log(r_2/r_1)}
				= \lh. \log\pfrac{\xs_*-\xs}{\xs_*-1-\xd} \middle/ \log\pfrac{\xs_*-\xQ-\xd}{\xs_*-1-\xd} \rh..
		\]
	Considering the numerator and denominator here we see that
		\begin{align*}
			\log\pfrac{\xs_*-\xs}{\xs_*-1-\xd} &= \log\lh(1+\fr{1-\xs+\xd}{\xs_*-1-\xd}\rh) = \fr{1-\xs+\xd}{\xs_*-1-\xd} + O\pfrac{1}{\xs_*^2}, \\
			\log\pfrac{\xs_*-\xQ-\xd}{\xs_*-1-\xd} &= \log\lh(1+\fr{1-\xQ}{\xs_*-1-\xd}\rh) = \fr{1-\xQ}{\xs_*-1-\xd} + O\pfrac{1}{\xs_*^2},
		\end{align*}
	whereby
		\[
			\xl = \fr{1-\xs+\xd+O(\xs_*^{-1})}{1-\xQ+O(\xs_*^{-1})} = \fr{1-\xs}{1-\xQ} + O(\xd) + O(\xs_*^{-1}).
		\]

	We thus have
		\[
			M_1^{1-\xl}M_2^{\xl} 
				\leq \pfrac{c_9}{\xd}^{1-\xl} \pth{\fr{c_{10}}{\xd}\log{t}}^{\xl}
				\leq \fr{c_{11}}{\xd}(\log{t})^{\xl},
		\]
	and letting $\xs_* = \xd^{-1} = \log\log{t}$ it follows that
		\<
			\label{eq:MrConvexBB}
			M_1^{1-\xl}M_2^{\xl} \leq c_{11} (\log\log{t}) (\log{t})^{\ellexp + O(\xd) + O(\xs_*^{-1})}.
		\>
	With our choices of $\xd$ and $\xs_*$ we have $(\log{t})^{O(\xd)} = e^{O(\xd\log\log{t})} = e^{O(1)} = O(1)$, and similarly for $(\log{t})^{O(\xs_*^{-1})}$. As $M_r \leq M_1^{1-\xl}M_2^\xl$ we conclude that
		\[
			|\log\xz(\xs+it)| \leq c_{10} (\log\log{t}) (\log{t})^{\ellexp} \qquad \lh(\xQ + \fr{1}{\log\log{t}} \leq \xs \leq 1\rh),
		\]
	and inequality \eqref{eq:lzCrudeBB2} follows for all $t > t_0$, giving the result.
\end{proof}

\bibliography{mubia.bib}
\end{document}